\documentclass[11pt,twoside]{amsart}
\usepackage{latexsym,amssymb,amsmath}
\usepackage[all]{xy}
\usepackage{tikz,pgfplots}
\usepackage{extpfeil}
\usepackage{afterpage}
\usepackage{float}
\usepackage{hyperref}

\makeatletter  % for 2020 subject classification
\@namedef{subjclassname@2020}{%
\textup{2020} Mathematics Subject Classification}
\makeatother

\numberwithin{equation}{section}
\newtheorem{theorem}{Theorem}[section]
\newtheorem{lemma}[theorem]{Lemma}
\newtheorem{proposition}[theorem]{Proposition}
\newtheorem{corollary}[theorem]{Corollary}

\def\cone{\operatorname{cone}}
\def\conv{\operatorname{conv}}

\theoremstyle{definition}
\newtheorem{definition}[theorem]{Definition}

\newtheorem{remark}[theorem]{Remark}
\newtheorem{example}[theorem]{Example}

\newcommand{\Z}{\mathbb{Z}}

\def\RR{{\mathbb{R}}}

\def\ZZ{{\mathbb{Z}}}
\def\NN{{\mathbb{N}}}

\begin{document}

\title[Ehrhart Functions of Weighted Lattice Points]
{Ehrhart Functions of Weighted Lattice Points}

\thanks{The first and fourth authors were supported by NSF grant
DMS-2434665, and the second and third authors were supported by SNII,
M\'exico.}  

\author[J. A. De Loera]{Jes\'us A. De Loera}
\address{
Department of Mathematics\\
University of California, Davis.
}
\email{deloera@math.ucdavis.edu}

\author[C. E. Valencia]{Carlos E. Valencia}
\address{
Departamento de
Matem\'aticas\\
Cinvestav, Av. IPN 2508, 07360, CDMX, M\'exico.
}
\email{cvalencia@math.cinvestav.edu.mx}

\author[R. H. Villarreal]{Rafael H. Villarreal}
\address{
Departamento de
Matem\'aticas\\
Cinvestav, Av. IPN 2508, 07360, CDMX, M\'exico.
}
\email{rvillarreal@cinvestav.mx}
%\thanks{*Corresponding author}

\author[C. Wang]{Chengyang Wang}
\address{
Department of Mathematics\\
University of California, Davis.
}
\email{cyywang@ucdavis.edu}

\keywords{Lattice polytopes, weights on lattice points, Ehrhart
functions and reciprocity laws, Stanley's positivity and monotonicity
theorems, Ehrhart rings, $h^*$-vectors, semigroup rings}  
\subjclass[2020]{Primary 52B20; Secondary 13F20, 05A15, 90C10} 

\begin{abstract}
This paper studies three different ways to assign weights to the
lattice points of a convex polytope and discusses the algebraic and
combinatorial properties of the resulting weighted Ehrhart functions,
their generating functions, and their associated rings.  
These will be called $q$-weighted, $r$-weighted, and $s$-weighted
Ehrhart functions, respectively. The key questions we investigate are:  
\emph{When are the weighted Ehrhart series rational functions, and
which classical Ehrhart theory properties are preserved?  
And when are the abstract formal power series the Hilbert series of
Ehrhart rings of some polytope?} 
We prove generalizations about weighted Ehrhart $h^*$-coefficients of
$q$-weighted Ehrhart series and show $q$- and $s$-weighted Ehrhart
reciprocity theorems.   
Then, we show the $q$- and $r$-weighted Ehrhart rings are the
(classical) Ehrhart rings of weight lifting polytopes. 
\end{abstract}

\maketitle 

\emph{Dedicated to Richard P. Stanley, on the occasion of his 80th
birthday and gratefully celebrating his contributions at the
intersection of commutative algebra and combinatorics.}  

\section{Introduction}\label{section-intro}

Ehrhart functions play an important role in algebraic combinatorics, commutative
algebra, and convex geometry. Formally, the \textit{Ehrhart function}
of a rational convex polytope 
$\mathcal{P}$ in $\mathbb{R}^s$ is defined by
$E_\mathcal{P}(n)=|n\mathcal{P}\cap\mathbb{Z}^s|$. A famous result
says $E_\mathcal{P}(n)$ is a \emph{quasipolynomial} of degree
$d=\dim(\mathcal{P})$, whose leading coefficient is the relative
volume ${\rm vol}(\mathcal{P})$ of $\mathcal{P}$ for integral
polytopes. Its generating
function $F_\mathcal{P}$ is a rational function and there is a nice
\emph{reciprocity law} between the Ehrhart function of $\mathcal{P}$
and that of its relative interior.
See~\cite{BeckRobins,BG-book,ehrhart-book,graphs,Sta2,monalg-rev} and
references therein for all details on these classical facts.          

While the classical Ehrhart function simply counts the lattice points. 
In recent years, there have been several natural suggestions on how
to extend Ehrhart functions by adding weights to the lattice points.  
This paper studies three different (but related) ways to assign
weights to the lattice points and the properties of the
resulting weighted Ehrhart functions, generating functions, and 
associated rings.    

Here are the key definitions for the rest of the paper.
In what follows $K$ is a field and $ S=K[t_1^{\pm 1},\ldots,t_s^{\pm 1}] $ is a
 Laurent polynomial ring over $K$. Monomials of $S$ are abbreviated
 by $t^a:=t_1^{a_1}\cdots t_s^{a_s}$, for $a=(a_1,\dots,a_s)$ in $\ZZ^s$.   

\begin{definition}
We say a function $w: {\ZZ}^s \rightarrow {\ZZ}$ is a \emph{weight}
of degree $e$ if for all vectors $v \in \ZZ^s$ and any scalar $k$,
$w(kv)=h(k)w(v)$ where $h(k)$ is a univariate polynomial of degree $e$.   
If $\mathbb{Z}\subset K$ and $w$ is extended to $w\colon
K^s\rightarrow K$, we also call $w$ a weight.  
\end{definition}

Examples of weight functions appear everywhere; here are a few we
consider: $w$ is a constant function, $w$ is a homogeneous linear
form, $w=w(t_1,\ldots,t_s)$ is a homogeneous polynomial, and $w$ is a
quasipolynomial.   

With weight functions in hand, we can put weights on the lattice
points (we count them with a weight) or, equivalently, on the
monomials of the Ehrhart ring.   

For a convex rational polytope $\mathcal{P}$ in $\mathbb{R}^s$, the \textit{cone over}  $\mathcal{P}$, denoted by ${\rm cone}(\mathcal{P})$, is the polyhedral cone
$\mathbb{R}_+\mathcal{B}_P$ in
$\mathbb{R}^{s+1}$ generated by  
\begin{equation}\label{nov2-24}
B_\mathcal{P}:= \{(x,1) : x \in \mathcal{P}\cap\mathbb{Z}^s\}\subset
\ZZ^{s+1}.
\end{equation}

\begin{definition}\label{weightedseries} 
Let $\mathcal{P}$ be a polytope. 
For $i = 1, \ldots, p$, let $w_{i} \colon K^s\rightarrow K$,
$a\mapsto w_{i}(a)$, be a weight such that $w_i(\ZZ^s)\subset\ZZ$,
and let $w\colon K^s\rightarrow K^p$ be the   
function $w=(w_1,\ldots,w_p)$.  
We now define weighted Ehrhart functions using the $w_{i}$'s and new
auxiliary variables $q_1,\ldots,q_p$:   
\begin{enumerate}
\item[(1)] The \textit{$q$-weighted multivariate Ehrhart function}
and \textit{$q$-weighted multivariate Ehrhart series} of
$\mathcal{P}$ relative to $w$, denoted $E_{\mathcal{P}}^{q,w}$ and
$F_{\mathcal{P}}^{q,w}$ respectively, are given by       
\begin{align*}
E_{\mathcal{P}}^{q,w}(q, t, n):=\hspace{-5mm} \sum_{\scriptstyle
(a,n)\in\cone(\mathcal{P})\cap \ZZ^{s+1}} 
\hspace{-5mm} q^{w(a)}t^{a},\quad F_{\mathcal{P}}^{q,w}(q, t,
x):=\sum_{n=0}^\infty E_\mathcal{P}^{q,w}(q, t, n) x^n,  
\end{align*}
where $q^{w(a)}=q_{1}^{w_{1}(a)}\cdots q_{p}^{w_{p}(a)}$. 
        
\item[(2)] The \textit{$r$-weighted multivariate Ehrhart function}
and \textit{$r$-weighted multivariate Ehrhart series} of
$\mathcal{P}$ relative to $w$, denoted $E_\mathcal{P}^{r,w}$ and
$F_{\mathcal{P}}^{r,w}$ respectively, are given by       
\begin{align*}
E^{r,w}_\mathcal{P}(q, t, n):=\hspace{-7mm}
\sum_{\begin{array}{c}\scriptstyle (a,n)\in\cone(\mathcal{P})\cap
\ZZ^{s+1}\vspace{-1mm}\\   
\scriptstyle 0 \leq b_i\leq w_i(a)\ \forall\,
i\end{array}}\hspace{-7mm} t^{a}q_1^{b_1}\cdots q_p^{b_p},\quad
F_\mathcal{P}^{r,w}(q, t, x):= \sum_{n=0}^\infty
E_\mathcal{P}^{r,w}(q, t, n)x^n.    
\end{align*}
\end{enumerate}

Finally, when there is only one single weight function $w\colon
K^s\rightarrow K$, $a\mapsto w(a)$. We define the last kind of
weighting using $w$ directly:  

\begin{enumerate}
\item[(3)] The \textit{$s$-weighted  Ehrhart function} and the
\textit{$s$-weighted multivariate Ehrhart series} of $\mathcal{P}$
relative to $w$, denoted $E_\mathcal{P}^{s,w}$ and
$F_\mathcal{P}^{s,w}$ respectively, are given by       
\[
E^{s,w}_\mathcal{P}(t,n):=\sum_{\scriptstyle(a,n)\in\cone(\mathcal{P})\cap
\ZZ^{s+1}}\hspace{-7mm} w(a)t^a,\  \mbox{ and }
F_\mathcal{P}^{s,w}(t,x) :=\sum_{n=0}^\infty
E_\mathcal{P}^{s,w}(t,n)x^n.
\]
\end{enumerate}
\end{definition}

\begin{remark}
The above definitions are very general, but
note that we recover old versions directly via substitutions and
specializations.  
For the $q$-weighted multivariate Ehrhart function and series when we
set $t_j=1$: 
\begin{align*}
E_\mathcal{P}^{q,w}(q, 1, n)=\hspace{-5mm}\sum_{\scriptstyle
(a,n)\in\cone(\mathcal{P})\cap \ZZ^{s+1}} \hspace{-5mm}
q^{w(a)},\quad F_\mathcal{P}^{q,w}(q, 1, x) =\sum_{n=0}^\infty
E_\mathcal{P}^{q,w}(q, 1, n) x^n.      
\end{align*}
This type of weighting was investigated in \cite{Chapoton} (but only for
$w_i$ linear and $p=1$).  
Note further that when we set $q_j=1$, we recover the traditional
Ehrhart series.    
Regarding the $s$-weighting definition, if we set $t_j=1$, we recover
the univariate versions introduced in \cite{Brion-Vergne}, 
$E^{s,w}_\mathcal{P}(n):=\sum_{\scriptstyle(a,n)\in\cone(\mathcal{P})\cap
\ZZ^{s+1}} w(a),\  \mbox{ and }\
F_\mathcal{P}^{s,w}(x):=\sum_{n=0}^\infty E_\mathcal{P}^{s,w}(n)x^n$.   
The definition of $r$-weighting is new, but as we will see, it connects the other two definitions.
\end{remark}

Moreover, we later prove the following:

\noindent \textbf{Proposition~\ref{Prop:q_imply_r&s}}\textit{ 
Both $r$-weighted multivariate Ehrhart series and $s$-weighted
Ehrhart series can be recovered from the multivariate $q$-weighted
Ehrhart series.   
}

The questions we discuss in this paper are:
\emph{When are the weighted Ehrhart series rational functions? What
classical properties are preserved? And, when are these abstract
formal power series the Hilbert series of Ehrhart rings of some polytopes?}

\subsection*{Our Contributions:}
Our contributions have two themes: first, the combinatorics of the generating functions and the weighted Ehrhart series, and second, the construction of graded rings whose Hilbert functions are precisely the weighted Ehrhart functions.

\subsubsection*{\bf Positivity of $\mathbf{h}^*$-coefficients for q-weighted Ehrhart series.} 
Two famous results in Ehrhart theory due to Stanley are the
\textit{positivity} theorem showing that the numerator of the
rational function representing the Ehrhart series of a lattice
polytope is a polynomial $h^*(x)$ with nonnegative coefficients
\cite{Stanley-nonneg-h-vector}, and the \textit{monotonicity} theorem
for $h^*$-vectors showing that given lattice polytopes
$\mathcal{P}_1$ and $\mathcal{P}_2$ with
$\mathcal{P}_1\subset\mathcal{P}_2$, then their $h^*$-vectors satisfy
$h^*(\mathcal{P}_1)\leq h^*(\mathcal{P}_2)$ component-wise
\cite{stanley-mono}.   
Could it be that the weighted Ehrhart series we consider are rational
functions and that Stanley's results extend to the weighted Ehrhart
series? We give answers to these questions showing some
rationality and positivity-type results but we do not prove monotonicity-type theorems for
weighted Ehrhart series.

It was shown recently that Stanley's positivity and monotonicity theorems \cite{Stanley-nonneg-h-vector,stanley-mono} continue to hold for $s$-weighted Ehrhart series of rational polytopes where lattice points are counted with homogeneous polynomial weights that are sums of products of linear forms which are nonnegative on the polytope \cite{jdl3}.     

Inspired by the fact that the numerator of the Ehrhart series always has nonnegative coefficients, Chapoton investigated in \cite{Chapoton} the numerator of univariate $q$-weighted Ehrhart series and noticed that such nonnegativity result does not hold anymore.    
In Section~\ref{weighted-ehrhart-series-section}, we take a deeper look and give a sufficient, but not necessary, geometric condition to show when the nonnegativity result holds.   
\begin{definition}
For a lattice polytope $\mathcal{P} \subset \RR^{s}$ and $p$ linear functions $w_{i}\colon \RR^s \to \RR\ (i = 1, \ldots, p)$, we say a triangulation $\mathcal{T}$ of the polytope $\mathcal{P}$ is $(w_{1},\ldots,w_{p})$-compatible if every simplex $\Delta \in \mathcal{T}$ satisfies that the multiset of weight vectors on the vertices of $\Delta$ is identical. 
\end{definition}

We come to one of our positivity results.

\noindent \textbf{Theorem~\ref{compatible-theorem}.} 
\textit{If $w_{i}$'s are linear weights and a lattice polytope $\mathcal{P}$ has a $(w_{1},\ldots,w_{p})$-compatible triangulation $\mathcal{T}$, then the numerator of the rational form of its $q$-weighted Ehrhart series has positive coefficients.}  

The condition of Theorem~\ref{compatible-theorem} is only a
sufficient condition for positivity (see Example~\ref{chengyang2}).

\subsubsection*{\bf Reciprocity for $q$ and $s$ weighted Ehrhart series.}
Another famous result in Ehrhart theory, also due to Stanley, is the \textit{multivariate law of reciprocity}, an elegant functional relation showing that when inverting each variable in the integer point transform of a rational cone \cite{Beck-Ehrhart}, \cite[p.~60]{BeckRobins}, one obtains (up to sign) the integer point transform for the interior of the cone \cite{Stanley-Reciprocity}. 
The classical Ehrhart-Macdonald law of reciprocity relating the Ehrhart function and the interior Ehrhart function follows as a specialization \cite[p.~84]{BeckRobins}. 
Inspired by this and Chapoton's reciprocity result for $q$-weighted Ehrhart series for only one weight, in Section~\ref{weighted-ehrhart-series-section}, we investigate the reciprocity properties for other weighted Ehrhart series we considered.         

\begin{definition}
For a polytope $\mathcal{P}$, we define: 
\begin{enumerate}
\item the \textit{interior $q$-weighted multivariate Ehrhart series} as
$$
F_{\mathcal{P}^{\circ}}^{q,w_{1}, \ldots, w_{p}}(q, t, x) := \sum_{(a,n)\in {\cone(\mathcal{P})}^{\circ}\cap \ZZ^{s+1} } q_{1}^{-w_{1}(-a)}\cdots q_{p}^{-w_{p}(-a)}t^{a}x^{n},  
$$
\item the \textit{interior $q$-weighted Ehrhart series} as
$$
F_{\mathcal{P}^{\circ}}^{q,w_{1}, \ldots, w_{p}}(q, 1, x) := \sum_{(a,n) \in {\cone(\mathcal{P})}^{\circ}\cap \ZZ^{s+1} } q_{1}^{-w_{1}(-a)}\cdots q_{p}^{-w_{p}(-a)}x^{n},\mbox{ and } 
$$
\item the \textit{interior $s$-weighted Ehrhart series} as
$$
F_{\mathcal{P}^{\circ}}^{s,w}(t,x) := \sum_{(a,n) \in {\rm cone}(\mathcal{P})^{\circ}\cap \ZZ^{s+1} } w(-a)t^ax^{n}. 
$$
\end{enumerate}
\end{definition}

We show the following reciprocity law for the Ehrhart series in the case of the multivariate $q$-weighting of the Ehrhart series.

\noindent \textbf{Theorem~\ref{rlaw}.}
\textit{If $w_1,\ldots,w_p$ are linear weights and $\mathcal{P}$ is a polytope whose cone has dimension $s + 1$, then the $q$-weighted multivariate Ehrhart series satisfies the reciprocity property, i.e., 
$$ 
F^{q,w_{1},\ldots,w_{p}}_{\mathcal{P}}(q^{-1}, t^{-1}, x^{-1}) =(-1)^{s+1}F^{q,w_{1},\ldots,w_{p}}_{\mathcal{P}^{\circ}}(q, t, x). 
$$
} 

As a consequence, $F_{\mathcal{P}}^{q, w_{1}, \ldots, w_{p}}(q^{-1}, t^{-1}, x^{-1})$ can be represented by interior $q$-weighted multivariate Ehrhart series (Corollary~\ref{nov19-23}).

We also show the following reciprocity law in the case of the $s$-weighting of Ehrhart functions.

\noindent \textbf{Theorem~\ref{Thm:s-weightedreciprocity}.}
\textit{Let $\mathcal{P} \subset \mathbb{R}^{s}$ be a rational polytope and $h(a)=\prod_{i=1}^{s+1}\sum_{j=1}^{k_{i}}P_{ij}(a_{i})\gamma_{ij}^{a_{i}}$, where the $P_{ij}$'s are polynomials and $\gamma_{ij}$ are nonzero complex numbers. 
Then   
\begin{enumerate}
\item $F_{\mathcal{P}}^{s, h}\left(t, x \right)$ and $F_{\mathcal{P}^{\circ}}^{s,h}\left(t, x \right)$ are rational $s$-weighted multivariate Ehrhart series, and
\item they satisfy the reciprocity relation, 
\[
F_{\mathcal{P}}^{s, h}\left(t^{-1}, x^{-1} \right) =  (-1)^{s+1} F_{\mathcal{P}^{\circ}}^{s, h}\left(t, x \right).
\]
\end{enumerate}
}

\subsection*{Weighted Ehrhart rings and weight lifting polytopes}

\begin{definition}\label{wer-def} 
Let $\mathcal{P}$ be a lattice polytope in $\mathbb{R}^s$ and let
$w_{i}\colon \RR^{s} \to \RR\, (i = 1,\ldots, p)$ be $p$ weight
functions such that $w_{i}(e_{j}) \in \NN$ for $j = 1, \ldots, s$.   
We define \textit{weighted Ehrhart-rings}:
\begin{enumerate}
\item The $q$-\textit{weighted Ehrhart ring} of $\mathcal{P}$, 
denoted $A_q^w(\mathcal{P})$, is a monomial subring  given by
$$ 
A_q^w(\mathcal{P}) :=K[t^aq^{w(a)}z^{n} \mid (a,n)\in {\rm
cone}(\mathcal{P})\cap\mathbb{Z}^{s+1}]\subset S[q,z],    
$$
where $w=(w_1,\ldots,w_p)$ and $q^{w(a)}=q_{1}^{w_{1}(a)}\cdots
q_{p}^{w_{p}(a)}$. We allow the case $t_j=1$. 
\item The \textit{$r$-weighted Ehrhart ring} of $\mathcal{P}$,
denoted $A_{r}^w(\mathcal{P})$, is a monomial subring given by  
$$ 
A_{r}^w(\mathcal{P}):=K[t^{a}q_{1}^{b_{1}}\cdots q_{p}^{b_{p}}z^{n}
\mid (a,n)\in {\rm cone}(\mathcal{P})\cap\mathbb{Z}^{s+1},\, 0 \leq
b_{i} \leq w_{i}(a)\mbox{ for all }i].  
$$
\end{enumerate}
\end{definition}
 Note that by making $w_i=0$ for all $i$, $A_q^w(\mathcal{P})$
and $A_r^w(\mathcal{P})$ are equal to the classical Ehrhart ring
$A(\mathcal{P})$ of $\mathcal{P}$. If $w_1,\ldots,w_p$ are linear
functions, then the weighted Ehrhart rings are finitely generated
graded $K$-algebras (Theorems~\ref{jdl} and \ref{weighted-ehrhart}).
As we will see in Proposition \ref{chengyang}, the hypothesis that
the $w_i$'s are linear is essential to prove that
$A_q^w(\mathcal{P})$ is finitely generated.

Two of our main results describe the $q$- and $r$-weighted Ehrhart rings
using classical Ehrhart rings (Theorems~\ref{jdl} and \ref{weighted-ehrhart}). 
The strategy is to construct new polytopes $\mathcal{P}^w$ and
$\mathcal{P}_w$ (see Eqs.~\eqref{nov1-24} and \eqref{dec11-22-2}, and
Example~\ref{hope-figure}) in
a higher dimension such that the $q$- and $r$-weighted Ehrhart rings
and the classic Ehrhart rings $A(\mathcal{P}^w)$ and $A(\mathcal{P}_w)$
are correspondent, that is,
$$A_q^w(\mathcal{P})=A(\mathcal{P}^w)\ \mbox{ and }\ 
A_r^w(\mathcal{P})=A(\mathcal{P}_w).
$$       
\quad We will call the polytopes $\mathcal{P}^w$ and
$\mathcal{P}_w$ the \textit{$q$- and $r$-weight lifting polytopes}, respectively. 
The same type of $r$-weight lifting polytopes are studied in \cite[Theorem~1.1]{jdl4}.   
The difference in representing the $r$-weight lifting polytope is
that they use linear inequalities, and we use vertices.

In Section \ref{r-weighted-section}, we relate the Ehrhart
function of $\mathcal{P}_w$ with $s$-weighted Ehrhart functions when
the weight $w$ is a product $w_1\cdots w_p$ of linear functions
(Theorem~\ref{weighted-ehrhart}(d)). We determine the
dimension of $\mathcal{P}_w$ when $w$ is a monomial of $S$ 
(Proposition~\ref{nov6-22}). As a consequence, if $\mathcal{P}$ is
full-dimensional, that is $\dim(\mathcal{P})=s$, then the degree of
$E_{\mathcal{P}_w}$ is $\dim(\mathcal{P})+\deg(w)$.

\subsubsection*{\bf $\mathbf{s}$-weighted Ehrhart functions and Ehrhart series} 
In Section~\ref{s-weighted-section}, we give other applications of 
Theorem~\ref{weighted-ehrhart} and recover some results from the literature. 
If $w$ is a monomial of $S$ of degree $p$, then $E_\mathcal{P}^{s,w}$ is a
polynomial whose leading coefficient is ${\rm vol}({\mathcal{P}_w})$ and 
$$
\deg(E^{s, w}_\mathcal{P})=\dim(\mathcal{P})+p,
$$
see Proposition~\ref{nov8-22}.
If $f$ is a polynomial, writing $f$ as a $K$-linear combination of monomials, this result allows us to find $E_\mathcal{P}^{s, f}$ and $F_\mathcal{P}^{s, f}$ using  polynomial interpolation (Corollary~\ref{nonneg-poly}(a)).   

We say that $\mathcal{P}={\rm conv}(v_1,\ldots,v_m)$ is \textit{non-degenerate} if for each $1\leq i\leq s$, there is $v_j$ such that the $i$-th entry of $v_j$ is non-zero. 

\noindent \textbf{Corollary~\ref{nonneg-poly}.}
\textit{ 
Let $f$ be a non-zero polynomial of $K[t_1,\ldots,t_s]$ of degree $p$ and let $d$ be the dimension of $\mathcal{P}$.
The following holds:
\begin{enumerate}
\item[(a)] $E_\mathcal{P}^{s, f}$ is a polynomial of $n$ of degree at most $d+p$ and $F_\mathcal{P}^{s,f}(x)$ is a rational function.
\item[(b)] If $\mathcal{P}$ is non-degenerate and $f$ is a monomial, then $\deg(E_\mathcal{P}^{s, f})=d+p$. 
\item[(c)] $E_\mathcal{P}^{s, f}$ is a $K$-linear combination of Ehrhart polynomials. 
\item[(d)]\cite[Proposition~4.1]{Brion-Vergne} If the interior $\mathcal{P}^{\rm o}$ of $\mathcal{P}$ is nonempty, $K=\mathbb{R}$, and $f$ is homogeneous and $f\geq 0$ on $\mathcal{P}$, then $E_\mathcal{P}^{s, f}$ is a polynomial of degree $s+p$.    
\end{enumerate}
}

It is known that in part (d), the leading coefficient of $E_\mathcal{P}^{s, f}$ is equal to the integral $\int_\mathcal{P}f$.
This fact appears in \cite[p.~437]{baldoni-etal} and \cite[Proposition~5]{Bruns-Soger}.
Integrals of the type $\int_\mathcal{P}f$ with $f$ a polynomial and $\mathcal{P}$ a rational polytope were studied in \cite{baldoni-etal-1,barvinok,barvinok1,Bruns-Soger,lasserre}.
Algorithms and software implementation to compute this integral were developed
independently in \textit{Normaliz} \cite{normaliz2} and
\textit{LattE integrale} \cite{latte-integrale}. 

We compare $E_\mathcal{P}^{s, f}(n)$ with $E_\mathcal{P}(n)$
when $f$ is a homogeneous polynomial of degree $p$ in $S$ and
$K=\mathbb{R}$ (Proposition~\ref{ineq-ehrhart}). If $w$ is a linear
function, we use Theorem~\ref{weighted-ehrhart} linking classical and weighted Ehrhart
theories, together with Stanley's positivity and monotonicity
theorems for $h^*$-vectors of polytopes
\cite{Stanley-nonneg-h-vector,stanley-mono} to show that
$E_\mathcal{P}^{s, w}(n)$ is the difference of two classical Ehrhart 
functions and that $F_\mathcal{P}^{s, w}(x)$ is a rational
function with positive $h^*$-vector
(Corollary~\ref{linear-weighted-ehrhart}). 

In later sections, we revisit the classical results in Ehrhart theory
and see some that extend to weighted multivariate Ehrhart series
(Theorems~\ref{rlaw} and \ref{Thm:s-weightedreciprocity}).

\subsubsection*{\bf Prior work on weighted Ehrhart theory:}

We are certainly not the first to introduce or study weights.
The theory of weighted Ehrhart functions and weighted Ehrhart 
series has been developed in several papers since at least the 1990s.   
See
\cite{baldoni-etal-1,baldoni-etal,Brion-Vergne,bruns-ichim-soger,Bruns-Soger,Chapoton} 
and references therein, as well as recent articles
\cite{Bajo,jdl3,Reiner}.  In 1997 Brion and Vergne \cite{Brion-Vergne} presented a generalization of Ehrhart's theorem in the context of Euler-Maclaurin
formulas where the points are counted with ``$s$-weights'' given by a
polynomial function $w$, i.e., $E^{s,w}_\mathcal{P}(n)=\sum_{a \in
n\mathcal{P} \cap \Z^s} w(a)$.        
Since then, it has been known that when $w$ is a polynomial with integer
coefficients and $\mathcal{P}$ is a lattice polytope, the
$s$-weighted Ehrhart function is a polynomial and the $s$-weighted
Ehrhart series is the power series expansion of a rational function.    
See \cite{Brion-Vergne} and Corollary~\ref{nonneg-poly}. 
Here, the $s$-weighted Ehrhart functions will be studied in Section~\ref{s-weighted-section}. 

To the best of our knowledge, research on $q$-weightings is less
expansive, and the first paper on the subject comes from Chapoton,
who introduced the $q$-weighted Ehrhart functions \cite{Chapoton} in
one variable thinking of $q$-analogues of the Ehrhart counting methods.   
Chapoton only looked at the $q$-weights of degree $1$, thus when 
$w$ is linear. The recent paper of Reiner and Rhoades \cite{Reiner} 
also looks at $q$-analogues of the Ehrhart series that come from 
deformations of the ring of polynomials modulo the ideal of polynomials 
that vanish at the lattice points of a polytope. Like us, they care about constructing rings whose Hilbert series is exactly their new Ehrhart series. 
We seem to be the first to introduce $r$-weightings.      

For any unexplained terminology and additional information, we refer to \cite{BeckRobins,BG-book} for Ehrhart theory and \cite{BHer,Sta1,Sta5,Sta2} for commutative algebra and enumerative combinatorics.   

\section{Weighted Ehrhart series}\label{weighted-ehrhart-series-section}

We start with some useful general observations.
In the introduction, we saw that from a polytope $\mathcal{P}$, we can assign a weight(s) to the lattice points; the weights are either a number, a monomial, a polynomial, etc., depending on the method of weighting. 
From it, we saw in Definitions \ref{weightedseries} that the method of weighting gives different formal power series, but all weighted multivariate Ehrhart series can be computed from a single one.

\begin{proposition}\label{Prop:q_imply_r&s}
Both $r$-weighted multivariate Ehrhart series and $s$-weighted Ehrhart series can be recovered from $q$-weighted multivariate Ehrhart series. 
\end{proposition}
\begin{proof}

Note that one can rewrite the $r$-weighted multivariate Ehrhart function as follows:
\begin{align*}
E^{r,w}_\mathcal{P}(q, t, n):=\hspace{-12mm} \sum_{\begin{array}{c}\scriptstyle (a,n)\in\cone(\mathcal{P})\cap \ZZ^{s+1}\vspace{-1mm}\\ 
\scriptstyle 0 \leq b_i\leq w_i(a)\ \forall\, i\end{array}}\hspace{-12mm} t^{a}q_1^{b_1}\cdots q_p^{b_p}=\hspace{-3mm}  \sum_{\scriptstyle a\in n\mathcal{P}\cap \ZZ^{s}\vspace{-1mm}}\bigg[ \sum_{\scriptstyle 0 \leq b_i\leq w_i(a)\ \forall\, i}\hspace{-7mm}q_1^{b_1}\cdots q_p^{b_p} \bigg]t^{a}= \sum_{\scriptstyle a\in n\mathcal{P}\cap \ZZ^{s}\vspace{-1mm}}\bigg[\prod_{i=1}^{p}\bigg( \sum_{j=0}^{w_{i}(a)}q_{i}^{j} \bigg) \bigg]t^{a}.
\end{align*}

Next, note that 
$$
\prod_{i=1}^{p}\left( \sum_{j=0}^{w_{i}(a)}q_{i}^{j} \right) t^{a}x^{n} = \prod_{i=1}^{p}\left(\frac{1-q_{i}^{w_{i}(a)+1}}{1-q_{i}} \right) t^{a}x^{n} =\sum_{I \subseteq [p]} \left( \frac{\prod_{j \in I} (-q_{j})}{\prod_{i=1}^{p}(1-q_{i})} \right) \left(\prod_{j \in I}q_{j}^{w_{j}(a)} \right) t^{a}x^{n},     
$$
where we allow the case $I=\emptyset$ and by convention, the product over an empty set is $1$. For any fixed index set $I$, one has 
$$
\left(\prod_{j \in I}q_{j}^{w_{j}(a)} \right) t^{a}x^{n} =\left. q_{1}^{w_{1}(a)}\cdots q_{p}^{w_{p}(a)}t^{a}x^{n} \right|_{q_{k} = 1, k \notin I}.
$$  
Therefore, $F_\mathcal{P}^{r,w_{1},\ldots,w_{p}}(q, t, x) =\sum_{I \subseteq [p]} \left( \frac{\prod_{j \in I} (-q_{j})}{\prod_{i=1}^{p}(1-q_{i})} \right) \left( \left.F_\mathcal{P}^{q,w_{1},\ldots,w_{p}}(q, t, x)\right|_{q_{k}= 1, k \notin I} \right)$, that is,
\begin{align}
F_\mathcal{P}^{r,w_{1},\ldots,w_{p}}(q, t, x)&
=\frac{1}{\prod_{i=1}^p(1-q_i)}\sum_{k=0}^p\Big({\sum_{1\leq
j_1<\cdots< j_k\leq p}}(-q_{j_1})\cdots(-q_{j_k})F_\mathcal{P}^{q,
w_{j_1},\ldots, w_{j_k}}(q,t,x)\Big).\nonumber
\end{align}
The first term in the summation, i.e., the term corresponding to $k=0$
is 
$$
F_\mathcal{P}^{q,0}(q,t,x)=\sum_{n=0}^\infty E_\mathcal{P}^{q,0}(q,t,n)x^n=\sum_{n=0}^\infty\Big(\hspace{1mm} \sum_{\scriptstyle a\in n\mathcal{P}\cap \ZZ^{s}\vspace{-1mm}}t^a\Big)x^n.
$$
Note that
$$
\left. \left[ q \cdot \frac{\partial}{\partial q}\left(q^{w(a)}t^{a}x^{n}\right) \right] \right|_{q=1} = \left. w(a)q^{w(a)}t^{a}x^{n}\right|_{q = 1} = w(a)t^ax^{n},
$$
therefore, $F_{\mathcal{P}}^{s,w}(t,x) = \left.\left[ q \cdot \frac{\partial}{\partial q} \left( F_{\mathcal{P}}^{q,w}(q, t, x)\right) \right] \right|_{q=1}$.
\end{proof}

As the following result shows, the hypothesis in Theorem~\ref{jdl} that $w_i$ is linear for $i=1,\ldots,p$ is essential to prove that the $q$-weighted graded monomial algebra is finitely generated.

\begin{proposition}\label{chengyang} 
Let $\mathcal{P}=\{1\}$ and $w(a) = a^{2}$, then the ring 
$$
A^{w}_{q}(\mathcal{P}) =  K[q^{a^2}x^a\mid a\in\mathbb{N}]
$$
is not Noetherian. 
In particular, $A^{w}_{q}(\mathcal{P})$ is not finitely generated as a $K$-algebra.
\end{proposition}
\begin{proof} 
A monomial $q^m x^n$ is in $A^{w}_{q}(\mathcal{P})$ if and only if there is a partition $\lambda$ of $n$, denoted $\lambda\vdash n$, such that the sum of the squares $\lambda_1^2+\lambda_2^2+\cdots$ is equal to $m$. 
Consider the following ideals of the ring $A^{w}_{q}(\mathcal{P})$
$$
I_k=(q^{i^2}x^i\mid 1\leq i\leq k),\ k\geq 1.
$$
\quad It suffices to show that $I_{k-1}\subsetneq I_k$ for $k\geq 2$.
We claim that $q^{k^2}x^k\in I_{k}\setminus I_{k-1}$. 
We argue by contradiction assuming that $q^{k^2}x^k\in I_{k-1}$. 
Then, 
$$
q^{k^2} x^k=(q^m x^n)(q^{i^2}x^i) 
$$ 
for some $q^m x^n\in A^{w}_{q}(\mathcal{P})$ and $1\leq i\leq k-1$. 
Hence, there is a partition $(n_1,\ldots,n_\ell)$ of $n$ such that 
\begin{align*} 
&n=n_1+\cdots+n_\ell,\ m=n_1^2+\cdots+n_\ell^2,\ n_i\in\mathbb{N},\\
&k=n+i=n_1+\cdots+n_\ell+i,\ k^2=m+i^2=n_1^2+\cdots+n_\ell^2+i^2,\quad \therefore\\
&(n_1+\cdots+n_\ell+i)^2=n_1^2+\cdots+n_\ell^2+i^2.
\end{align*}
\quad Hence, from the last equality, we get $n_i=0$ for $i=1,\ldots\ell$, $n=0, m=0$, and consequently $q^{k^2}x^k=q^{i^2}x^i$ for some $1\leq i\leq k-1$, a contradiction.
\end{proof}

When weights are defined by linear functions, we can conclude that all the weighted Ehrhart series we considered are rational functions, just as in the traditional case.   

\begin{proposition}\label{proposition: linearweight_rational_imply_rational}
If each weight $w_{i}(a) = v_{i}^{\intercal}a + b_{i}$ is linear, then the $q$-weighted multivariate Ehrhart series, the $q$-weighted Ehrhart series, the $r$-weighted multivariate Ehrhart series, the $r$-weighted Ehrhart series and the $s$-weighted multivariate Ehrhart series are all rational functions.
\end{proposition}
\begin{proof}
The multivariate Ehrhart series of a polytope $\mathcal{P}$ is
$$
\sum_{\scriptstyle (a,n)\in\cone(\mathcal{P})\cap \ZZ^{s+1}} t^{a}x^{n}.
$$
Recall that $t^a$ is the abbreviation of $t_{1}^{a_{1}}\cdots t_{s}^{a_{s}}$. 
We can apply (and in fact effectively compute) the following monomial substitutions: $t_{1} \mapsto q_{1}^{v_{1,1}} \cdots q_{p}^{v_{p,1}} t_{1}$, $\ldots$, $t_{s} \mapsto q_{1}^{v_{1,s}} \cdots q_{p}^{v_{p,s}} t_{s}$. Then, by the linearity of weights $w_{i}$'s,
$$
t^{a}x^{n} \mapsto q_{1}^{v_{1}^{\intercal}a}\cdots q_{p}^{v_{p}^{\intercal}a}t^{a}x^{n}.
$$
\quad Lastly, we can just multiply the series by $q_{1}^{b_{1}} \cdots q_{p}^{b_{p}}$.
The initial multivariate Ehrhart series has a rational form, then the monomial substitutions gives another rational form for the $q$-weighted multivariate Ehrhart series. 
Hence, by Proposition \ref{Prop:q_imply_r&s}, the rest of the weighted Ehrhart series are also rational functions.
\end{proof}

\subsection{Positivity of $\mathbf{h}^*$-vector}

We begin this section by giving a sufficient condition for the positivity of $h^*$-coefficients. 
Then, we give extensions of Ehrhart's reciprocity law.

\begin{theorem}\label{compatible-theorem}
If $w_{i}$'s are linear weights and a lattice polytope $\mathcal{P}$ has a $(w_{1},\ldots,w_{p})$-compatible triangulation $\mathcal{T}$, then the numerator of the rational form of its $q$-weighted Ehrhart series has positive coefficients.   
Nevertheless, this is only a sufficient condition for positivity.
\end{theorem}
\begin{proof}
We can construct a disjoint partition of $\mathcal{P}$ using the triangulation $\mathcal{T}$, $\mathcal{P} = \sqcup_{\Delta \in \mathcal{T}} \Delta^{*}$ with $\Delta^{*}$ being obtained possibly by
removing several facets of $\Delta$. 
Note that the $q$-weighted Ehrhart function is additive with respect to disjoint union. 
Therefore,  
$$
E_{\mathcal{P}}^{q, w_{1}, \ldots, w_{p}}(q, 1, n) = \sum_{\Delta \in \mathcal{T} } E_{{\Delta^{*}}}^{q, w_{1}, \ldots, w_{p}}(q, 1, n),  
$$ 
and similarly,
$$
F_{\mathcal{P}}^{q, w_{1}, \ldots, w_{p}}(q, 1, n) = \sum_{\Delta \in \mathcal{T} }  F_{{\Delta^{*}}}^{q, w_{1}, \ldots, w_{p}}(q, 1, n).  
$$
\quad Since $\Delta^{*}$ is a simplex with several facets possibly removed, using \cite[Theorem~3.5]{BeckRobins}, it is easy to see that its $q$-weighted Ehrhart series has the following rational form
$$
F_{{\Delta^{*}}}^{q, w_{1}, \ldots, w_{p}}(q, 1, n) = \frac{h_{\Delta^{*}}(q, x) }{ \prod_{i=1}^{d+1}(1 - q_{1}^{w_{1}(v_{i})}\cdots q_{p}^{w_{p}(v_{i})}x)},
$$
with $v_{i}$ being the vertices of $\Delta^{*}$ and $h_{\Delta^{*}}(q, x) \in \mathbb{N}[q, x]$.

Since $\mathcal{T}$ is $(w_{1},\ldots,w_{p})$-compatible, the denominator is same for every $\Delta \in T$, without loss of generality, we can denote it as $\prod_{i=1}^{d+1}(1 - q^{\alpha_{i}}x)$, 
$$
F_{\mathcal{P}}^{q, w_{1}, \ldots, w_{p}}(q, 1, n) = \frac{ \sum_{\Delta \in \mathcal{T}}h_{\Delta^{*}}(q, x) }{ \prod_{i=1}^{d+1}(1-q^{\alpha_{i}}x) }.  
$$  
\quad To see this rational form is reduced, we can simply degenerate $q=1$ and use the classical Ehrhart theory. 
\end{proof}

\begin{example}\label{chengyang2}
Consider the polytope $\mathcal{P} = \conv((0, 0, 0), (1, 0, 0),(1, 1, 0), (1, 1, 1), (2, 1, 1))$ with the linear function $w(y_1, y_2, y_3) =y_1 + y_2 + y_3$, we find that its $q$-weighted Ehrhart series is:
\[ 
\frac{1+q^{2}x}{(1-x)(1-qx)(1-q^{3}x)(1-q^{4}x)},
\]
and the polytope $\mathcal{P}$ does not have any $w$-compatible triangulation.
\end{example}

\subsection{Reciprocity for weighted Ehrhart series}
\begin{theorem}\label{rlaw} 
If $w_1,\ldots,w_p$ are linear weights and $\mathcal{P}$ is a polytope whose cone has dimension $s+1$, then the $q$-weighted multivariate Ehrhart series satisfies the reciprocity property, i.e., 
$$ 
F^{q,w_{1},\ldots,w_{p}}_{\mathcal{P}}(q^{-1}, t^{-1}, x^{-1}) =(-1)^{s+1}F^{q,w_{1},\ldots,w_{p}}_{\mathcal{P}^{\circ}}(q, t, x). 
$$
\end{theorem} 
\begin{proof} 
By assumption ${\rm cone}(\mathcal{P})$ has dimension $s+1$. 
Apply Stanley's reciprocity theorem for rational cones \cite[Theorem~4.3]{BeckRobins} to $\cone(\mathcal{P})$, which states that
$$ 
\sum_{(a,n) \in \cone(\mathcal{P}) \cap \ZZ^{s+1}} \left(t_{1}^{-1}\right)^{a_{1}} \cdots \left( t_{s}^{-1}\right)^{a_{s}}\left(x^{-1}\right)^{n}  = (-1)^{s+1}\sum_{(a,n) \in \cone(\mathcal{P})^{\circ} \cap \ZZ^{s+1}} t_{1}^{a_{1}} \cdots t_{s}^{a_{s}}x^{n},    
$$
apply the appropriate monomial substitutions: $t_{i} \mapsto q_{1}^{w_1(e_i)} \cdots q_{p}^{w_p(e_i)} t_{i}$ for $i = 1, \ldots,s$.
For $(a,n)\in \cone(\mathcal{P}) \cap \ZZ^{s+1}$, $(a,n)=( a_{1}, \ldots, a_{s}, n )$, by the linearity of weights $w_{i}$'s, we can see   
$$
\left( t_{1}^{-1}\right)^{a_{1}} \cdots \left( t_{s}^{-1}\right)^{a_{s}} \left(x^{-1}\right)^{n} \mapsto \left( q_{1}^{-1}\right)^{w_1(a)}\cdots \left( q_{p}^{-1}\right)^{w_p(a)} \left( t_{1}^{-1}\right)^{a_{1}} \cdots \left( t_{s}^{-1}\right)^{a_{s}} \left( x^{-1}\right)^{n}.     
$$

Similarly for $(a,n) \in \cone(\mathcal{P})^{\circ} \cap \ZZ^{s+1}$,
$$
t_{1}^{a_{1}} \cdots t_{s}^{a_{s}}x^{n} \mapsto q_{1}^{w_1(a)}\cdots q_{p}^{w_p(a)}t_{1}^{a_{1}} \cdots t_{s}^{a_{s}}x^n,
$$
and the asserted equality follows readily.
\end{proof}

\begin{corollary}\label{nov19-23}
$F_{\mathcal{P}}^{r, w_{1}, \ldots, w_{p}}(q^{-1}, t^{-1}, x^{-1})$ can be represented by interior $q$-weighted multivariate Ehrhart series. 
\end{corollary}
\noindent \textit{Proof.} 
This follows from the equalities
\begin{align*}
F_\mathcal{P}^{r,w_{1},\ldots,w_{p}}(q^{-1}, t^{-1}, x^{-1}) &= \sum_{I \subseteq [p]} \left( \frac{\prod_{j \in I} (-q_{j}^{-1})}{\prod_{i=1}^{p}(1-q_{i}^{-1})} \right) \left( \left.F_\mathcal{P}^{q,w_{1},\ldots,w_{p}}(q^{-1}, t^{-1}, x^{-1})\right|_{q_{k} = 1, k \notin I} \right) \\    
&= (-1)^{s+1} \sum_{I \subseteq [p]} \left( \frac{\prod_{j \in I} (-q_{j}^{-1})}{\prod_{i=1}^{p}(1-q_{i}^{-1})} \right) \left( \left.F^{q,w_{1}, \ldots,w_{p}}_{\mathcal{P}^{\circ}}(q, t, x)\right|_{q_{k} = 1, k \notin I} \right).\Box
\end{align*}

\quad Now we look at reciprocity in the case of the $s$-weighting of the Ehrhart series. 

Let $\mathcal{P}$ be a simplex in $\mathbb{R}^s$ with vertices $v_1,\ldots,v_{s+1}$. 
We simplify notation for the rest of this section, we identify $v_i$ with $(v_i,1)$. 
With this identification $v_1,\ldots,v_{s+1}$ is a basis of $\mathbb{R}^{s+1}$. 
We denote by $\phi$, the projection $\mathbb{R}^{s+1}\mapsto\mathbb{R}$, $a=(a_i)\mapsto a_{s+1}$.       
Thus, from $\phi(z)$, we recover the dilation level of $z$.

\begin{lemma}[Naive Weighted Reciprocity Lemma]\label{Lem:NaiveReciprocity} 
Assume the following conditions. 
\begin{itemize}
\item The $\mathcal{P}$ is a simplex, i.e., the vertices of $\mathcal{P}$, $\{v_{i}\}$, forms a basis of $\RR^{s+1}$.  
\item $h(a)$ is separable and multiplicative with respect to the basis $\{v_{i}\}$, i.e., there exist univariate functions $h_{i}$ such that $h(a) = h_{1}(\alpha_{1})\cdots h_{s+1}(\alpha_{s+1})$ with $a = \alpha_{1}v_{1} + \cdots + \alpha_{s+1}v_{s+1}$.
\item $h_{i}(x) = \sum_{j = 1}^{k_{i}}P_{ij}(x)\gamma_{ij}^{x}$, where the $P_{ij}$'s are polynomials and $\gamma_{ij}\in\mathbb{C}\setminus\{0\}$.  
\end{itemize}
Then, $F_{\mathcal{P}}^{s, h}\left(t, x \right)$ and $F_{\mathcal{P}^{\circ}}^{s,h}\left(t, x \right)$ are rational $s$-weighted multivariate Ehrhart series, and they satisfy the reciprocity relation, 
$$
F_{\mathcal{P}}^{s, h}\left(t^{-1}, x^{-1} \right) =  (-1)^{s+1} F_{\mathcal{P}^{\circ}}^{s, h}\left(t, x \right).
$$
\end{lemma}

\noindent \textit{Proof.} 
Denote $\Diamond  = \{\sum c_{i}v_{i}: 0 \leq c_{i} <1\}$ as the fundamental parallelepiped generated by $\{v_{i}\}$.
For any $a \in  \cone{(\mathcal{P})} \cap \mathbb{Z}^{s+1}$, we can represent $ a = \alpha_{1}v_{1} +\cdots + \alpha_{s+1}v_{s+1}$. 
Decompose it into integer parts and fractional parts, we have 
$$ 
a = \underbrace{\lfloor \alpha_{1} \rfloor v_{1} +\cdots + \lfloor \alpha_{s+1} \rfloor v_{s+1} }_{\text{integral parts}} + \underbrace{ \{\alpha_{1}\}v_{1} + \cdots + \{\alpha_{s+1}\}v_{s+1}}_{\text{fractional parts}}.
$$
For simplicity, we denote the fractional parts as $r \in \Diamond$. 
Then
\begin{align*}
F_{\mathcal{P}}^{s,h}(t, x) &= \sum_{a \in  \cone{(\mathcal{P})} \cap \mathbb{Z}^{s+1}} h(a)t^{a} x^{\phi(a)} \\
&= \sum_{\alpha_{1}v_{1} +\cdots + \alpha_{s+1}v_{s+1} \in \cone{(\mathcal{P})} \cap \mathbb{Z}^{s+1}} \prod_{i=1}^{s+1}\left( h_{i}(\alpha_{i}) t^{\lfloor \alpha_{i} \rfloor v_{i}} x^{ \phi(\lfloor \alpha_{i} \rfloor v_{i} ) }\right) \cdot t^{r}x^{\phi(r)} \\
&= \sum_{r \in \Diamond}\sum_{n_{1}=0}^{\infty}\cdots\sum_{n_{s+1}=0}^{\infty}\prod_{i=1}^{s+1}\left( h_{i}(n_{i}+\{\alpha_{i}\})t^{n_{i} v_{i}} x^{ \phi( n_{i} v_{i} ) }\right) \cdot t^{r}x^{\phi(r)} \\
&= \sum_{\beta \in \Diamond} \prod_{i=1}^{s+1}\left(\sum_{n_{i}=0}^{\infty} h_{i}(n_{i}+\{\alpha_{i}\})t^{n_{i} v_{i}} x^{ \phi( n_{i} v_{i} ) } \right) \cdot t^{r}x^{\phi(r)}.
\end{align*}
Similarly, for any $a \in  \cone{(\mathcal{P}^{\circ})} \cap \mathbb{Z}^{s+1}$, we can decompose it into
$$
a = \underbrace{\lceil \alpha_{1} \rceil v_{1} +\cdots + \lceil \alpha_{s+1} \rceil v_{s+1} }_{\text{integral parts}} - \underbrace{ ( \{-\alpha_{1}\}v_{1} + \cdots + \{-\alpha_{s+1}\}v_{s+1} )}_{\text{fractional parts}}.
$$
Denote the fractional part as $r \in \Diamond$. 
Then
\begin{align*}
F_{\mathcal{P}^{\circ}}^{s,h}(t, x) &= \sum_{a \in  \cone{(\mathcal{P}^{\circ})} \cap \mathbb{Z}^{s+1}} h(-a)t^{a} x^{\phi(a)} \\
&= \sum_{\alpha_{1}v_{1} +\cdots + \alpha_{s+1}v_{s+1} \in \cone{(\mathcal{P}^{\circ})} \cap \mathbb{Z}^{s+1}}\prod_{i=1}^{s+1}\left( h_{i}(-\alpha_{i})t^{\lceil \alpha_{i} \rceil v_{i}} x^{\phi(\lceil \alpha_{i} \rceil v_{i}) }\right) \cdot t^{-r}x^{\phi(-r)} \\
&= \sum_{r \in \Diamond}\sum_{n_{1}=1}^{\infty}\cdots\sum_{n_{s+1}=1}^{\infty}\prod_{i=1}^{s+1}\left( h_{i}(-n_{i}+\{-\alpha_{i}\})t^{n_{i} v_{i}} x^{\phi(n_{i} v_{i})}\right) \cdot t^{-r}x^{\phi(-r)} \\
&= \sum_{r \in \Diamond} \prod_{i=1}^{s+1}\left(\sum_{n_{i}= 1}^{\infty} h_{i}(-n_{i}+\{-\alpha_{i}\}) t^{n_{i} v_{i}} x^{\phi(n_{i} v_{i})} \right) \cdot t^{-r}x^{\phi(-r)}.
\end{align*}
According to the assumption and the equivalent characterizations of rational power series. 
For each $r \in \Diamond$, generating functions
$$
G_{i}(t, x) := \sum_{n_{i}=0}^{\infty} h_{i}(n_{i}+\{\alpha_{i}\}) t^{n_{i} v_{i}} x^{ \phi( n_{i} v_{i} ) } \quad\text{and}\quad \overline{G_{i}}(t, x)  := \sum_{n_{i}= 1}^{\infty}h_{i}(-n_{i}+\{-\alpha_{i}\}) t^{n_{i} v_{i}} x^{ \phi( n_{i} v_{i} ) }
$$
are rational. 
In particular, they satisfy the reciprocity property, i.e.,
$$
G_{i}(t^{-1}, x^{-1}) = (-1)\overline{G_{i}}(t, x).
$$
Therefore,
\begin{align*}
F_{\mathcal{P}}^{s,h}(t^{-1}, x^{-1})  &= \sum_{r \in \Diamond} \prod_{i=1}^{s+1}\left( G_{i}(t^{-1}, x^{-1}) \right) \cdot (t^{-1})^{r}(x^{-1})^{\phi(-r)} \\
&= \sum_{r \in \Diamond} \prod_{i=1}^{s+1}\left((-1)\overline{G_{i}}(t, x) \right) \cdot t^{-r}x^{\phi(-r)} =(-1)^{s+1} F_{\mathcal{P}^{\circ}}^{s,h}(t, x). \quad \Box
\end{align*}

\begin{lemma}\label{transf}
Assume $h(a) =\prod_{i=1}^{s+1}\sum_{j=1}^{k_{i}}P_{ij}(a_{i})\gamma_{ij}^{a_{i}}$, where the $P_{ij}$'s are polynomials and $\gamma_{ij}$ are nonzero complex numbers, then $h(a)$ can be decomposed into finite terms where each term is separable and multiplicative with respect to any basis $\{v_{i}\}$ of $\RR^{s+1}$.   
\end{lemma}
\noindent \textit{Proof.}
Only need to prove when $h(a) = a_{1}^{m_{1}}\cdots a_{s+1}^{m_{s+1}} \cdot \gamma_{1}^{a_{1}}\cdots \gamma_{s+1}^{a_{s+1}}$.    
Suppose $a = \alpha_{1}v_{1} + \cdots + \alpha_{s+1}v_{s+1}$, then we can simply replace $a_{i}$ by $\alpha_{1}v_{1, i}+ \cdots + \alpha_{s+1}v_{s+1, i}$ to get
\begin{align*}
h(a) &= \prod_{i=1}^{s+1}( \alpha_{1}v_{1, i}+ \cdots + \alpha_{s+1}v_{s+1, i} )^{m_{i}} \cdot \prod_{i=1}^{s+1} \gamma_{i}^{\alpha_{1}v_{1, i}+ \cdots + \alpha_{s+1}v_{s+1, i}} \\
&=\sum_{\mathbf{k}}c_{\mathbf{k}}\alpha^{\mathbf{k}}\prod_{i=1}^{s+1}(\gamma_{i}^{v_{1,1}}\cdots \gamma_{d}^{v_{i,s+1}})^{\alpha_{i}}=\sum_{\mathbf{k}}c_{\mathbf{k}}\alpha^{\mathbf{k}} \prod_{i=1}^{s+1} \hat{\gamma}_{i}^{\alpha_{1}}.\quad \Box 
\end{align*}

\begin{theorem}\label{Thm:s-weightedreciprocity}
Let $\mathcal{P} \subset \mathbb{R}^{s}$ be a rational polytope and $h(a)=\prod_{i=1}^{s+1}\sum_{j=1}^{k_{i}}P_{ij}(a_{i})\gamma_{ij}^{a_{i}}$, where the $P_{ij}$'s are polynomials and $\gamma_{ij}$ are nonzero complex numbers. 
Then, $F_{\mathcal{P}}^{s, h}\left(t, x \right)$ and $F_{\mathcal{P}^{\circ}}^{s,h}\left(t, x \right)$ are rational $s$-weighted multivariate Ehrhart series, and they satisfy the reciprocity relation, 
$$
F_{\mathcal{P}}^{s, h}\left(t^{-1}, x^{-1} \right) =  (-1)^{s+1} F_{\mathcal{P}^{\circ}}^{s, h}\left(t, x \right).
$$  
\end{theorem}
\begin{proof}
Using the fact that every rational polyhedral cone can be triangulated into simplicial cones and the inclusion-exclusion principle, we can assume that $\mathcal{P}$ is simplicial. 
Then, applying Lemma \ref{transf}, we can assume $h(a)$ is separable and multiplicative with respect to a basis. 
Finally, we can apply Naive Weighted Reciprocity Lemma \ref{Lem:NaiveReciprocity}.     
\end{proof}

\section{$q$-weighted Ehrhart rings}\label{q-weighting}
Let $S=K[t_1,\ldots,t_s]$ be a polynomial ring over a field $K$
containing $\mathbb{R}$ as a subfield, let $v_1,\ldots,v_m$ be points
in $\mathbb{N}^s$, let $\mathcal{P}={\rm conv}(v_1,\ldots,v_m)$, let
$w_i\colon\mathbb{R}^s\rightarrow\mathbb{R}$, $i=1,\ldots,p$, be    
linear functions such that $w_i(e_j)\in\mathbb{N}$ for all $i,j$, and
let $w\colon\mathbb{R}^s\rightarrow\mathbb{R}^p$ be the linear map
$w=(w_1,\ldots,w_p)$.   

The $q$-\textit{weight lifting polytope} of $\mathcal{P}$, denoted by
$\mathcal{P}^w$, is given by 
\begin{equation}\label{nov1-24}
\mathcal{P}^w:={\rm
conv}(\{(v_i,w(v_i))\}_{i=1}^m)\subset\mathbb{R}^{s+p}.
\end{equation} 
The linear map $\psi\colon \mathbb{R}^s\rightarrow\mathbb{R}^{s+p}$,
$a\mapsto(a,w(a))$ induces a bijective map 
$$
n\mathcal{P}\cap\mathbb{Z}^s\stackrel{\psi}{\rightarrow}
n\mathcal{P}^w\cap \mathbb{Z}^{s+p}, \mbox{ for }n\in\mathbb{N}, 
$$
and consequently $E_\mathcal{P}(n)=E_{\mathcal{P}^w}(n)$ for $n\geq 1$. 
We will give a different explanation of this equality by showing that
$A_q^w(\mathcal{P})$ is the Ehrhart ring $A(\mathcal{P}^w)$ of
$\mathcal{P}^w$ (see Eq.~\eqref{jan20-23}).

\smallskip

We have the following theorem.

\begin{theorem}\label{jdl} 
Let ${\rm cone}(\mathcal{P})$ be the cone over $\mathcal{P}={\rm conv}(\{v_i\}_{i=1}^m)$. 
The following hold.

{\rm (a)} $A_q^w(\mathcal{P})=K[t^aq^{w(a)}z^n\mid (a,n)\in {\rm cone}(\mathcal{P})\cap\mathbb{Z}^{s+1}]= K[t^aq^{w(a)}z^n\mid a\in n\mathcal{P}\cap\mathbb{Z}^{s}]$.

{\rm (b)} $A_q^w(\mathcal{P})=\bigoplus_{n=0}^\infty A_q^w(\mathcal{P})_n$ is a graded $K$-algebra whose $n$-th graded component is 
\begin{align*}
A_q^w(\mathcal{P})_n:=\sum_{\scriptstyle a\in n\mathcal{P}\cap\mathbb{Z}^{s}}\hspace{-4mm}Kt^aq^{w(a)}z^n\mbox{ and }\dim_K(A_q^w(\mathcal{P})_n)=E_\mathcal{P}(n),\, \forall\, n\geq 0.  
\end{align*}

{\rm (c)} $A(\mathcal{P}^w)=A_q^w(\mathcal{P})$ and $E_\mathcal{P}(n)=E_{\mathcal{P}^w}(n)$ for $n\geq1$.

{\rm (d)} There is a finite set $\mathcal{H}_\mathcal{P}=
\{(c_i,d_i)\}_{i=1}^r$, $c_i\in\mathbb{Z}^s$, $d_i\in\mathbb{Z}$ such that ${\rm cone}(\mathcal{P})\cap\mathbb{Z}^{s+1}=\mathbb{N}\mathcal{H}_\mathcal{P}$, 
$$
A(\mathcal{P})=K[\mathbb{N}\mathcal{H}_\mathcal{P}], \mbox{ and } A(\mathcal{P}^w)=K[t^{a}q^{w(a)}z^n\mid a\in n\mathcal{P}\cap\mathbb{Z}^s]=K[\{t^{c_i}q^{w(c_i)}z^{d_i}\}_{i=1}^r].
$$
\quad {\rm (e)} The $q$-weighted multivariate Ehrhart series $F_\mathcal{P}^{q,w}(t,q,x)$ of $\mathcal{P}$ is a rational function.
\end{theorem}

\begin{proof} 
(a) Note that ${\rm cone}(\mathcal{P})$ is the cone generated by $\{(v_i,1)\}_{i=1}^m$. The equality follows noticing that $(a,n)\in{\rm cone}(\mathcal{P})\cap\mathbb{Z}^{s+1}$ if and only if $a\in n\mathcal{P}\cap\mathbb{Z}^s$.

(b) The inclusion $A_q^w(\mathcal{P})_kA_q^w(\mathcal{P})_n\subset A_q^w(\mathcal{P})_{k+n}$, $k,n\in\mathbb{N}$, follows from the linearity of $w=(w_1,\ldots,w_p)$ and the convexity of $\mathcal{P}$. 
As the set $\{t^aq^{w(a)}z^n\mid a\in n\mathcal{P}\cap\mathbb{Z}^s\}$ has $|n\mathcal{P}\cap\mathbb{Z}^s|$ elements and is linearly independent over $K$, one has $\dim_K(A_q^w(\mathcal{P})_n)=E_\mathcal{P}(n)$.

(c) To show that $A(\mathcal{P}^w)\subset A_q^w(\mathcal{P})$ take
$t^aq^bz^n\in A(\mathcal{P}^w)$.  
Then, $(a,b,n)\in{\rm cone}(\mathcal{P}^w)\cap\mathbb{Z}^{s+p+1}$, 
$$
(a,b,n)=\sum_{i=1}^m\lambda_i(v_i,w(v_i),1),\, \lambda_i\geq 0,\
(a,n)\in {\rm cone}(\mathcal{P})\cap\mathbb{Z}^{s+1}, 
$$ 
and $b=w(a)$. 
Thus, $t^aq^bz^n=t^aq^{w(a)}z^n$ and $(a,n)\in {\rm cone}(\mathcal{P})\cap\mathbb{Z}^s$, that is, $t^aq^bz^n\in A_q^w(\mathcal{P})$. 
To show that $A(\mathcal{P}^w)\supset A_q^w(\mathcal{P})$ take $t^aq^{w(a)}z^n\in A_q^w(\mathcal{P})$. 
Then, $(a,n)\in {\rm cone}(\mathcal{P})\cap\mathbb{Z}^{s+1}$,
$$
(a,n)=\sum_{i=1}^m\lambda_i(v_i,1),\, \lambda_i\geq 0,\, \mbox{ and } w(a)=\sum_{i=1}^m\lambda_iw(v_i).
$$
\quad Then, $(a,w(a),n)=\sum_{i=1}^m\lambda_i(v_i,w(v_i),1)\in{\rm
cone}(\mathcal{P}^w)\cap\mathbb{Z}^{s+p+1}$, and $t^aq^{w(a)}z^n\in A(\mathcal{P}^w)$. 
Finally, the equality $E_\mathcal{P}(n)=E_{\mathcal{P}^w}(n)$ for
$n\geq 0$ follows from 
$A(\mathcal{P}^w)=A_q^w(\mathcal{P})$ and part (b).

(d) The existence of $\mathcal{H}_\mathcal{P}$ satisfying ${\rm cone}(\mathcal{P})\cap\mathbb{Z}^{s+1}=\mathbb{N}\mathcal{H}_\mathcal{P}$ and $A(\mathcal{P})=K[\mathbb{N}\mathcal{H}_\mathcal{P}]$ follows from \cite[Theorem~9.3.6]{monalg-rev}. 
The other expressions for $A(\mathcal{P}^w)$ follow from parts (a) and (c). 

(e) By the proof of part (c), $(a,b,n)\in {\rm cone}({\mathcal{P}^w})\cap\mathbb{Z}^{s+p+1}$ if and only if $(a,n)\in {\rm cone}(\mathcal{P})\cap\mathbb{Z}^{s+1}$ and $b=w(a)$. 
Then, by \cite[Corollary~3.7]{BeckRobins}, the integer-point
transform of ${\rm cone}({\mathcal{P}^w})$ is a rational function given by
\begin{align*}
\frac{G(t,q,x)}{\prod_{i=1}^{\ell}(1-t^{\alpha_i}q^{\beta_i}x^{n_i})}=\sum_{\scriptstyle
(a,b,n)\in{\rm cone}({\mathcal{P}^w})\cap \ZZ^{s+p+1}} \hspace{-5mm} t^{a}q^bx^n  
=\sum_{\scriptstyle (a,n)\in {\rm cone}(\mathcal{P})\cap \ZZ^{s+1}}
\hspace{-5mm} t^{a}q^{w(a)}x^n=F_\mathcal{P}^{q,w}(t,q,x). 
\end{align*}
\quad Thus, $F_\mathcal{P}^{q,w}(t,q,x)$ is a rational function. 
\end{proof}

\section{$r$-weighted Ehrhart rings}\label{r-weighted-section}

In this section, we prove one of our main results 
and show some applications to weighted Ehrhart theory. 
Let $S=K[t_1,\ldots,t_s]$ be a polynomial ring over a field $K$ containing $\mathbb{R}$ as a subfield, let $v_1,\ldots,v_m$ be points in $\mathbb{N}^s$, let $\mathcal{P}={\rm conv}(v_1,\ldots,v_m)$, let $w_i\colon\mathbb{R}^s\rightarrow\mathbb{R}$, $i=1,\ldots,p$, be
linear functions such that $w_i(e_j)\in\mathbb{N}$ for all $i,j$, and
let $w\colon\mathbb{R}^s\rightarrow\mathbb{R}^p$ be the linear map
given by $w=(w_1,\ldots,w_p)$. 

There is a rich relationship between Ehrhart functions and
commutative algebra, which we briefly recall here (see
\cite{Reiner,Sta1,Sta2,monalg-rev} for all details).    
Let $\mathcal{P}$ be a lattice polytope in $\RR^s$ of dimension $d$
and let $B_\mathcal{P}$ be as in Eq.~\eqref{nov2-24}.   
We are interested in the finitely generated  affine semigroup
${\rm cone}(\mathcal{P})\cap\mathbb{Z}^{s+1}$, where
${\rm cone}(\mathcal{P})$ denotes the polyhedral cone
$\mathbb{R}_+B_\mathcal{P}$ generated by
$B_\mathcal{P}$.    

The
resulting monomial algebras are often called \emph{toric algebras}
and are core to the theory of toric varieties
\cite{BG-book,cox-toric-book,Stur1}; their generating sets come from
Hilbert bases, and their Hilbert functions are investigated by
several researchers \cite{BG-book,BHer,Sta1,Sta2}.      

On the algebraic side, the \textit{Ehrhart ring} of
$\mathcal{P}$ is the monomial ring of
$S[z]$ given by    
\begin{equation}\label{jan20-23}
A(\mathcal{P}):=K[t^az^n\mid a\in
n\mathcal{P}\cap\mathbb{Z}^s]=K[t^az^n\mid (a,n)\in 
{\rm cone}(\mathcal{P})\cap\mathbb{Z}^{s+1}]\subset S[z],   
\end{equation}
where $z$ is a \emph{grading} variable. 
The Ehrhart ring $A(\mathcal{P})$ has a natural $\ZZ$-grading given by  
\[
A(\mathcal{P})=\bigoplus_{n=0}^\infty A(\mathcal{P})_n,
\]
where $t^az^n\in A(\mathcal{P})_n$ if and only if $a\in
n\mathcal{P}\cap\mathbb{Z}^s$.
The $r$-weighted Ehrhart ring of $\mathcal{P}$ has the same graded
ring structure as $A(\mathcal{P})$. 
The ring $A_r^w(\mathcal{P})=\bigoplus_{n=0}^\infty
A_r^w(\mathcal{P})_n$ is a graded $K$-algebra whose $n$-th graded component is  
\begin{align}\label{eq-rgrading}
A_r^w(\mathcal{P})_n=\sum_{
\begin{array}{c}
\scriptstyle a\in n\mathcal{P}\cap\mathbb{Z}^{s}\vspace{-1mm}\\
\scriptstyle 0\leq b_i\leq w_i(a)
\end{array} 
}\hspace{-4mm}Kt^aq_1^{b_1}\cdots q_p^{b_p}z^n
\end{align}
and, by the following Lemma~\ref{jan7-23}, the dimension 
of $A_r^w(\mathcal{P})_n$ as a $K$-vector space is 
\begin{align*}
\dim_K(A_r^w(\mathcal{P})_n)&=E_\mathcal{P}^{s, \scriptstyle\prod_{i=1}^p(w_i+1)}(n),
\end{align*}
where $E_\mathcal{P}^{s, f}$ is the $s$-weighted Ehrhart function of
$f=\prod_{i=1}^p(w_i+1)$. 

\begin{lemma}\label{jan7-23} Let ${w}_i\colon K^s\rightarrow K$, $i=1,\ldots,p$, be linear functions such that $w_i(e_j)\in\mathbb{N}$ for all $i,j$. 
If $w=w_1\cdots w_p$ and $A_r^w(\mathcal{P})=\bigoplus_{n=0}^\infty A_r^w(\mathcal{P})_n$ is the $r$-weighted Ehrhart ring of $\mathcal{P}$, then
\begin{align*}
\dim_K(A_r^w(\mathcal{P})_n)=E_\mathcal{P}(n)+\Big(\sum_{i=1}^pE_\mathcal{P}^{s, w_i}(n)\Big) + \Big({\sum_{1\leq j_1< j_2\leq p}}E_\mathcal{P}^{s, w_{j_1}w_{j_2}}(n)\Big)+\cdots+E_\mathcal{P}^{s, w_1\cdots w_p}(n).\nonumber
\end{align*}
\end{lemma}
\begin{proof}  
The $n$-th graded component of $A_r^w(\mathcal{P})$ is given by
Eq.~\eqref{eq-rgrading}.
To determine the dimension of $A_r^w(\mathcal{P})_n$, note that for
each $a\in n\mathcal{P}\cap\mathbb{Z}^{s}$, the set  
$$
\{t^aq^bz^n\mid b=(b_1,\ldots,b_p)\in([0,w_1(a)]\times\cdots\times[0,w_p(a)])\cap\mathbb{Z}^{p}\}
$$
has $\prod_{i=1}^p(w_i(a)+1)$ elements. 
Then
\begin{align}
\dim_K(A_r^w(\mathcal{P})_n)&=\sum_{a\in n\mathcal{P}\cap\mathbb{Z}^{s}}\,\,\prod_{i=1}^p(w_i(a)+1)= E_\mathcal{P}^{s, \scriptstyle\prod_{i=1}^p(w_i+1)}(n)\nonumber\\
&=\sum_{k=0}^p\Big({\sum_{1\leq j_1<\cdots< j_k\leq p}}E_\mathcal{P}^{s, w_{j_1}\cdots w_{j_k}}(n)\Big)\nonumber\\
&=E_\mathcal{P}(n)+\Big(\sum_{i=1}^pE_\mathcal{P}^{s, w_i}(n)\Big)+\Big({\sum_{1\leq j_1< j_2\leq p}}E_\mathcal{P}^{s, w_{j_1}w_{j_2}}(n)\Big)+\cdots+E_\mathcal{P}^{s, w_1\cdots w_p}(n),\nonumber
\end{align}
and the proof is complete.
\end{proof}

The following lemmas will be used to show Theorem~\ref{weighted-ehrhart}.

\begin{lemma}\label{nov100} Let $\delta,\eta_1,\ldots,\eta_q$ be a sequence of nonnegative integers. 
If $\eta_1+\cdots+\eta_q\geq \delta$, then there are integers $\alpha_1,\ldots,\alpha_q$ such that $\alpha_1+\cdots+\alpha_q=\delta$ and $0\leq\alpha_i\leq \eta_i$ for all $i$.
\end{lemma}
\begin{proof}
We proceed by induction on $\delta\geq 0$. 
If $\delta=0$, we set $\alpha_i=0$ for $i=1,\ldots,q$. 
Assume that $\delta\geq 1$ and $\eta_1+\cdots+\eta_q\geq\delta$. 
Then, $\eta_j\geq 1$ for some $j$ and 
$$
\eta_1+\cdots+\eta_{j-1}+(\eta_j-1)+\eta_{j+1}+\cdots+\eta_q\geq\delta-1.
$$
\quad Then, by induction, there are integers $\alpha_1',\ldots,\alpha_q'$ such that $\delta-1=\alpha_1'+\cdots+\alpha_q'$, $0\leq\alpha_i'\leq \eta_i$ for $i\neq j$, and $\alpha_j'\leq \eta_j-1$. Setting $\alpha_i=\alpha_i'$ for $i\neq j$ and $\alpha_j=\alpha_j'+1$, we obtain $\delta=\alpha_1+\cdots+\alpha_q$ and $0\leq\alpha_i\leq \eta_i$ for $i=1,\ldots,q$.
\end{proof}

\begin{lemma}\label{oct22-22} 
Let $\delta,\epsilon,\mu$ be positive integers. 
If $\delta\leq\epsilon\mu$, then there are positive integers $\epsilon_i,\mu_i$, $i=1,\ldots,r$, such that $\delta=\sum_{i=1}^r\epsilon_i\mu_i$, $\epsilon\geq\sum_{i=1}^r\epsilon_i$, and $1\leq \mu_i\leq\mu$ for $i=1,\ldots,r$.
\end{lemma}
\begin{proof} We argue by induction on $\epsilon\geq 1$. 
If $\epsilon=1$, we set $\epsilon_1=1$, $\mu_1=\delta$, and $r=1$. 
Assume that $\epsilon>1$ and $\delta\leq\epsilon\mu$. 
If $\delta\leq \mu$, we set $\epsilon_1=1$, $\mu_1=\delta$, and $r=1$. 
Thus, we may assume that $\delta>\mu$. Then, $1\leq \delta-\mu\leq (\epsilon-1)\mu$ and, by induction, there are positive integers $\epsilon_i,\mu_i$, $i=1,\ldots,q$, such that $\delta-\mu=\sum_{i=1}^q\epsilon_i\mu_i$, $\epsilon-1\geq\sum_{i=1}^q\epsilon_i$, and $1\leq \mu_i\leq\mu$ for $i=1,\ldots,q$. 
To complete the induction process, we set $r=q+1$, $\epsilon_r=1$ and $\mu_r=\mu$. 
\end{proof}

\begin{lemma}\label{sep15-22} 
Let $r$ be a positive integer and let $n,\delta,\epsilon_1,\ldots,\epsilon_r,w_1,\ldots,w_r$ be positive real $($resp. rational$)$ numbers. 
If $\delta\leq\sum_{i=1}^rn\epsilon_iw_i$, then there are nonnegative real $($resp. rational$)$ numbers $\mu_1,\ldots,\mu_r$ such that $\delta=\sum_{i=1}^rn\epsilon_i\mu_i$ and $0\leq \mu_i\leq w_i$ for $i=1,\ldots,r$.
\end{lemma}
\begin{proof} 
We proceed by induction on $r$. 
If $r=1$, setting $\mu_1=\delta/(n\epsilon_1)$, one has $\delta=n\epsilon_1\mu_1$ and $\mu_1=\delta/(n\epsilon_1)\leq w_1$. 
Assume that $r>1$ and $\delta\leq\sum_{i=1}^rn \epsilon_i w_i$.  
If $\delta\leq n \epsilon_1 w_1$, setting $\mu_1=\delta/(n\epsilon_1)$ and $\mu_i=0$ for $i=2,\ldots,r$, one has $\delta=\sum_{i=1}^rn\epsilon_i\mu_i$ and $0\leq \mu_i\leq w_i$ for $i=1,\ldots,r$. 
If $\delta> n \epsilon_1 w_1$, then $0<\delta-n \epsilon_1 w_1\leq \sum_{i=2}^rn \epsilon_i w_i$. 
Hence, by induction, there are nonnegative numbers $\mu_2,\ldots,\mu_r$ such that $\delta-n \epsilon_1 w_1=\sum_{i=2}^rn\epsilon_i\mu_i$ and $0\leq \mu_i\leq w_i$ for $i=2,\ldots,r$. 
Then, the result follows by setting $\mu_1= w_1$.
\end{proof}

\begin{lemma}\label{image-cube} 
Let $w_1,\ldots,w_d$ be nonnegative real numbers and let $T\colon\mathbb{R}^d\mapsto\mathbb{R}^d$ be the linear map such that $T(e_i)=w_ie_i$ for $i=1,\ldots,d$. 
Then, $T([0,1]^d)=[0,w_1]\times\cdots\times[0,w_d]$.
\end{lemma}
\begin{proof} 
To show the inclusion ``$\subset$'' take $\alpha\in T([0,1]^d)$. 
Then, $\alpha=T(x_1,\ldots,x_d)$, $1\leq x_i\leq 1$ for all $i$, and consequently $\alpha=x_1w_1e_1+\cdots+x_dw_de_d$. 
Thus, $\alpha$ is in $[0,w_1]\times\cdots\times[0,w_d]$ because $0\leq x_iw_i\leq w_i$ for all $i$. 
Now, we show the inclusion ``$\supset$''. 
Take $\gamma\in[0,w_1]\times\cdots\times[0,w_d]$. 
Then, $\gamma=\gamma_1e_1+\cdots+\gamma_de_d$, $0\leq\gamma_i\leq w_i$ for all $i$. 
Consider the vector $\beta=(\beta_1,\ldots,\beta_d)$, given by $\beta_i=0$ if $w_i=0$ and $\beta_i=\gamma_i/w_i$ if $w_i>0$. 
Then, $\beta\in[0,1]^d$ and $\gamma=T(\beta)$, that is, $\gamma\in T([0,1]^d)$.
\end{proof}

As a preparation for introducing and proving our main result, 
we need to introduce some notation and recall some facts. By 
\cite[Theorem~9.3.6]{monalg-rev}, the Ehrhart ring of $\mathcal{P}$ is
given by  
\[
A(\mathcal{P})=K[\mathbb{Z}^{s+1}\cap{\rm
cone}(\mathcal{P})]=K[t^az^n\mid
(a,n)\in\mathbb{Z}^{s+1}\cap{\rm cone}(\mathcal{P})],
\]
and ${\rm cone}(\mathcal{P})$ has a Hilbert basis 
$\mathcal{H}_\mathcal{P}=\{(c_1,d_1),\ldots,(c_r,d_r)\}$,
$c_i\in\mathbb{N}^s$, $d_i\in\mathbb{N}$, that is, $\mathcal{H}_\mathcal{P}$ is a
finite subset of $\mathbb{N}^{s+1}$ and
${\rm cone}(\mathcal{P})\cap\mathbb{Z}^{s+1}=\mathbb{N}\mathcal{H}_\mathcal{P}$,
where $\mathbb{N}\mathcal{H}_\mathcal{P}$ is the semigroup spanned by
$\mathcal{H}_\mathcal{P}$.              
The Ehrhart ring of $\mathcal{P}$ is the
\emph{semigroup ring} of the 
affine
semigroup $\mathbb{N}\mathcal{H}_\mathcal{P}$, that is,
\[
A(\mathcal{P})=K[\mathbb{N}\mathcal{H}_\mathcal{P}]=K[t^{c_1}z^{d_1},\ldots,t^{c_r}z^{d_r}].
\]
\quad We construct from $\mathcal{P}$ and the linear weights
$w_1,\ldots,w_p$ a 
new lattice polytope $\mathcal{P}_w$ that we call the $r$-\textit{weight
lifting polytope} of $\mathcal{P}$. 
\begin{align}
\mathcal{G}:=&\{(v_i,0)+w_{j_1}(v_i)e_{s+j_1} +\cdots+w_{j_k}(v_i)e_{s+j_k}\mid 1\leq i\leq m,\ 1\leq j_1<\cdots< j_k\leq p\},\label{dec11-22-1}\\
&\mbox{we allow }k=0\mbox{ by setting } w_{j_1}(v_i)e_{s+j_1} +\cdots+w_{j_k}(v_i)e_{s+j_k}:=0\mbox{ if }k=0,\nonumber\\
&\mbox{that is},\mbox{ if }k=0,(v_i,0)\in\mathcal{G}\mbox{ for }i=1,\ldots,m, \nonumber\\
\mathcal{P}_w:=&{\rm conv}(\mathcal{G}),\quad \Gamma:=B_\mathcal{G}=\{(\alpha,1)\mid \alpha\in\mathcal{G}\},\label{dec11-22-2}\\  
\mathcal{H}_{\mathcal{P}_w}:=&\{(c_i,b,d_i)\mid b=(b_1,\ldots,b_p)\in\mathbb{N}^p,\, 0\leq b_j\leq w_j(c_i)\ \forall\, j,\, c_i\in d_i\mathcal{P},\ 1\leq i\leq r\}.
\label{dec11-22-3}
\end{align}

The following theorem is one of our main results about $r$-weighted
Ehrhart rings. In particular, it shows that
$A_r^w(\mathcal{P})$ is the Ehrhart ring $A(\mathcal{P}_w)$ of
$\mathcal{P}_w$ when $w$ is a monomial of $S$.  

\begin{theorem}\label{weighted-ehrhart} 
Let $w_1,\ldots,w_p$ be linear functions such that $w_i(e_j)\in\mathbb{N}$ for all $i,j$ and let ${\mathcal{P}_w}$, $\mathcal{H}_{\mathcal{P}_w}$, $\Gamma$, be as in Eqs.~\eqref{dec11-22-1}-\eqref{dec11-22-3}.    
Then  
\begin{enumerate}
\item[(a)]
$\mathbb{R}_+\Gamma\cap\mathbb{Z}^{s+p+1}=\mathbb{N}\mathcal{H}_{\mathcal{P}_w}$, that is, $\mathcal{H}_{\mathcal{P}_w}$ is a Hilbert basis of $\mathbb{R}_+\Gamma$,  
\item[(b)] $A_r^w(\mathcal{P})=K[\mathbb{N}\mathcal{H}_{\mathcal{P}_w}]$, $w=w_1\cdots w_p$,
\item[(c)] $A({\mathcal{P}_w})=A_r^w(\mathcal{P})$, and
\item[(d)]
$E_{\mathcal{P}_w}(n)=E_\mathcal{P}^{s,\scriptstyle\prod_{i=1}^p(w_i+1)}(n)$ for $n\geq 0$.  
\item[(e)] The $r$-weighted multivariate Ehrhart series
$F_\mathcal{P}^{r,w}(t,q,x)$ of $\mathcal{P}$ is a rational function. 
\end{enumerate}
\end{theorem}
\begin{proof} 
(a) To show the inclusion ``$\subset$'' we proceed by induction on $p\geq 1$. Assume that $p=1$. Take $(a,b,n)\in \mathbb{R}_+\Gamma\cap\mathbb{Z}^{s+2}$, $b\in\mathbb{N}_+$. 
Then
\begin{align}
(a,b,n)=&\sum_{i=1}^m\mu_{0,i}(v_i,0,1)+\sum_{i=1}^m\mu_{1,i}(v_i,w_1(v_i),1),
\end{align}
where $\mu_{i,j}\geq 0$, and consequently we obtain the following equalities
\begin{align}
a=&\mu_{0,1}v_1+\cdots+\mu_{0,m}v_m+\mu_{1,1}v_1+\cdots+\mu_{1,m}v_m,\label{oct21-22-1}\\
b=&\mu_{1,1}w_1(v_1)+\cdots+\mu_{1,m}w_1(v_m),\label{oct21-22-2}\\
n=&\mu_{0,1}+\cdots+\mu_{0,m}+\mu_{1,1}+\cdots+\mu_{1,m},\nonumber\\
(a,n)=&\sum_{i=1}^m\mu_{0,i}(v_i,1)+\sum_{i=1}^m\mu_{1,i}(v_i,1)=\eta_1(c_1,d_1)+\cdots+\eta_r(c_r,d_r),
\label{oct21-22-4}
\end{align}
for some $\eta_1,\ldots,\eta_r\in \mathbb{N}$. 
The last equality follows recalling that ${\rm cone}(\mathcal{P})\cap\mathbb{Z}^{s+1}=\mathbb{N}\mathcal{H}_\mathcal{P}$ and noticing that $(a,n)\in{\rm cone}(\mathcal{P})\cap\mathbb{Z}^{s+1}$. 
From Eqs.~\eqref{oct21-22-1} and \eqref{oct21-22-2}, one has
$$
b=\sum_{i=1}^m\mu_{1,i}w_1(v_i)\leq \sum_{i=1}^m\mu_{0,i}w_1(v_i)+\sum_{i=1}^m\mu_{1,i}w_1(v_i)=w_1(a),
$$
that is, $b\leq w_1(a)$. 
Hence, from Eq.~\eqref{oct21-22-4}, we get
\begin{equation}\label{touselater}
b\leq w_1(a)=\eta_1w_1(c_1)+\cdots+\eta_rw_1(c_r).
\end{equation}
\quad As $\eta_i,w_1(c_i)\in\mathbb{N}$ for $i=1,\ldots,r$, by Lemma~\ref{nov100}, we can write $b=b_{1,1}+\cdots+b_{1,r}$ for some $b_{1,i}$'s in $\mathbb{N}$ such that $b_{1,i}\leq \eta_iw_1(c_i)$ for $i=1,\ldots,r$. 
If $b_{1,i}=0$, we set $k_{1,i}=1$ and $\eta_{1,i}=u_{1,i}=0$. 
Then, applying Lemma~\ref{oct22-22} to each inequality $b_{1,i}\leq \eta_iw_1(c_i)$ with $b_{1,i}>0$, there are nonnegative integers $\eta_{1,i}^{(j)},u_{1,i}^{(j)}$ such that
\begin{align}
b_{1,i}&=\sum_{j=1}^{k_{1,i}}\eta_{1,i}^{(j)}u_{1,i}^{(j)},\quad \eta_i\geq\sum_{j=1}^{k_{1,i}}\eta_{1,i}^{(j)},\ u_{1,i}^{(j)}\leq w_1(c_i),
\label{oct21-22-5}
\end{align}
for $i=1,\ldots,r$. 
Writing $\eta_i=(\sum_{j=1}^{k_{1,i}}\eta_{1,i}^{(j)})+(\eta_i-\sum_{j=1}^{k_{1,i}}\eta_{1,i}^{(j)})$, by Eqs.~\eqref{oct21-22-4}--\eqref{oct21-22-5}, we get
\begin{align*}
(a,b,n)=&\sum_{i=1}^r\eta_i(c_i,0,d_i)+\Big(\sum_{i=1}^rb_{1,i}\Big)e_{s+1}\\
=&\sum_{i=1}^r\Big(\sum_{j=1}^{k_{1,i}}\eta_{1,i}^{(j)}\Big)(c_i,0,d_i)+\Big(\sum_{i=1}^r\Big(\sum_{j=1}^{k_{1,i}}\eta_{1,i}^{(j)}u_{1,i}^{(j)}\Big)\Big)e_{s+1}+\sum_{i=1}^r\Big(\eta_i-\sum_{j=1}^{k_{1,i}}\eta_{1,i}^{(j)}\Big)(c_i,0,d_i)\\
=&\sum_{i=1}^r\Big(\sum_{j=1}^{k_{1,i}}\eta_{1,i}^{(j)}(c_i,0,d_i)+\sum_{j=1}^{k_{1,i}}\eta_{1,i}^{(j)}u_{1,i}^{(j)}e_{s+1}\Big)+\sum_{i=1}^r\Big(\eta_i-\sum_{j=1}^{k_{1,i}}\eta_{1,i}^{(j)}\Big)(c_i,0,d_i)\\
=&\sum_{i=1}^r\Big(\sum_{j=1}^{k_{1,i}}\eta_{1,i}^{(j)}(c_i,u_{1,i}^{(j)},d_i) \Big)+\sum_{i=1}^r\Big(\eta_i-\sum_{j=1}^{k_{1,i}}\eta_{1,i}^{(j)}\Big)(c_i,0,d_i),
\end{align*}
that is, $(a,b,n)\in\mathbb{N}\mathcal{H}_{\mathcal{P}_w}$. 
Assume that $p>1$. 
We set 
\begin{align}
\mathcal{G}':=&\{(v_i,0)+w_{j_1}(v_i)e_{s+j_1}+\cdots+w_{j_k}(v_i)e_{s+j_k}\mid 1\leq i\leq m,\, 1\leq j_1<\cdots< j_k<p\},\nonumber\\
{\mathcal{P}}':=&{\rm conv}(\mathcal{G}'), \quad \Gamma':=\{(\alpha',1)\mid \alpha'\in\mathcal{G}'\},\nonumber\\
\mathcal{H}_{\mathcal{P}^{'}}:=&\{(c_i,b',d_i)\mid b'=(b_1,\ldots,b_{p-1})\in\mathbb{N}^{p-1},\, 0\leq b_j\leq w_j(c_i)\ \forall\, j,\, c_i\in d_i\mathcal{P},\ 1\leq i\leq r\}.
\label{oct29-22-1}
\end{align}
\quad Take $(a,b,n)\in\mathbb{R}_+\Gamma\cap\mathbb{Z}^{s+p+1}$, $b=(b_1,\ldots,b_p)$. 
Then, one has
\begin{align}
(a,b,n)=&\hspace{-5mm}\sum_{\begin{array}{c}
\scriptstyle 1\leq i\leq m\vspace{-1mm}\\
\scriptstyle 1\leq j_1<\cdots<j_k\leq p\end{array}}\hspace{-5mm}\mu_{i,j_1,\ldots,j_k}\big((v_i,0)+w_{j_1}(v_i)e_{s+j_1}+\cdots+w_{j_k}(v_i)e_{s+j_k}+e_{s+p+1}\big),
\end{align}
where $\mu_{i,j_1,\ldots,j_k}\geq 0$. 
Consequently, noticing that $b_i$ is the $(i+s)$-entry of $(a,b,n)$, we obtain the following equalities
\begin{align}
a=&\sum\mu_{i,j_1,\ldots,j_k}v_i,\label{oct26-22-1}\\
b_1e_{s+1}+\cdots+b_pe_{s+p}=&\sum\mu_{i,j_1,\ldots,j_k}(w_{j_1}(v_i)e_{s+j_1}+\cdots+w_{j_k}(v_i)e_{s+j_k}),\label{oct26-22-2}\\
n=&\sum\mu_{i,j_1,\ldots,j_k},\label{oct26-22-3}\\
(a,n)=&\sum\mu_{i,j_1,\ldots,j_k}(v_i+e_{s+p+1}).
\label{oct26-22-4}
\end{align}
\quad Consider the following Cartesian products
\begin{align}
&([0,w_1(c_1)]\times\cdots\times [0,w_{p-1}(c_1)])\cap\mathbb{Z}^{p-1}=\{\beta_{1,1},\ldots,\beta_{1,q_1}\},\label{oct29-22-4}\\
&\quad\quad\quad\quad\quad\quad\vdots\quad\quad\quad\quad\quad\quad \quad\quad\quad\quad\quad\vdots\nonumber\\
&([0,w_1(c_r)]\times\cdots\times [0,w_{p-1}(c_r)])\cap\mathbb{Z}^{p-1}=\{\beta_{r,1},\ldots,\beta_{r,q_r}\}.\label{oct29-22-5}
\end{align}
\quad Setting $b'=(b_1,\ldots,b_{p-1})$, one has $(a,b',n)\in\mathbb{R}_+\Gamma'\cap\mathbb{Z}^{s+p}$. 
Then, by induction, we get $(a,b',n)\in\mathbb{N}\mathcal{H}_{\mathcal{P}'}$. 
Therefore, we can write
\begin{equation}\label{oct26-22}
(a,b',n)=\sum_{i=1}^r\sum_{j=1}^{q_i}\eta_{i,j}(c_i,\beta_{i,j},d_i),
\end{equation}
where $\eta_{i,j}\in\mathbb{N}$. 
From Eqs.~\eqref{oct26-22-1} and \eqref{oct26-22-2}, one has
$$
b_p=\sum_{j_k=p}\mu_{i,j_1,\ldots,j_k}w_p(v_i)\leq \sum\mu_{i,j_1,\ldots,j_k}w_p(v_i)=w_p(a),
$$
that is, $b_p\leq w_p(a)$. 
Hence, from Eq.~\eqref{oct26-22}, we get
$$
b_p\leq w_p(a)=\sum_{i=1}^r\sum_{j=1}^{q_i}\eta_{i,j}w_p(c_i).
$$
\quad As $\eta_{i,j}w_p(c_i)\in\mathbb{N}$, by Lemma~\ref{nov100}, we can write
$$
b_p=\sum_{i=1}^r\sum_{j=1}^{q_i}\alpha_{i,j},
$$
where $\alpha_{i,j}\in\mathbb{N}$ and $\alpha_{i,j}\leq \eta_{i,j}w_p(c_i)$. 
Then, applying Lemma~\ref{oct22-22} to each inequality, there are nonnegative integers $\eta_{i,j}^{(\ell)},u_{i,j}^{(\ell)}$ such that
$$
b_p=\sum_{i=1}^r\sum_{j=1}^{q_i}\sum_{\ell=1}^{k_{i,j}}\eta_{i,j}^{(\ell)}u_{i,j}^{(\ell)},
$$
$u_{i,j}^{(\ell)}\leq w_p(c_i)$ and $\eta_{i,j}\geq\sum_{\ell=1}^{k_{i,j}}\eta_{i,j}^{(\ell)}$.
As $(a,b,n)=(a,b',n)+b_pe_{s+p}$, by Eq.~\eqref{oct26-22}, we get 
\begin{align*}
(a,b,n)=&\sum_{i=1}^r\sum_{j=1}^{q_i}\eta_{i,j}(c_i,\beta_{i,j},0,d_i)+b_pe_{s+p}=\sum_{i=1}^r\sum_{j=1}^{q_i}
\Big(\sum_{\ell=1}^{k_{i,j}}\eta_{i,j}^{(\ell)}\Big)(c_i,\beta_{i,j},0,d_i)+b_pe_{s+p}\\
&+\sum_{i=1}^r\sum_{j=1}^{q_i} \Big(\eta_{i,j}-\sum_{\ell=1}^{k_{i,j}}\eta_{i,j}^{(\ell)}\Big)(c_i,\beta_{i,j},0,d_i)\\
&=\sum_{i=1}^r\sum_{j=1}^{q_i} \sum_{\ell=1}^{k_{i,j}}\eta_{i,j}^{(\ell)}(c_i,\beta_{i,j},0,d_i)+\Big(\sum_{i=1}^r\sum_{j=1}^{q_i}\sum_{\ell=1}^{k_{i,j}} \eta_{i,j}^{(\ell)}u_{i,j}^{(\ell)}\Big)e_{s+p}\\
&+\sum_{i=1}^r\sum_{j=1}^{q_i} \Big(\eta_{i,j}-\sum_{\ell=1}^{k_{i,j}}\eta_{i,j}^{(\ell)}\Big)(c_i,\beta_{i,j},0,d_i)\\
&=\sum_{i=1}^r\sum_{j=1}^{q_i} \sum_{\ell=1}^{k_{i,j}}\eta_{i,j}^{(\ell)}(c_i,\beta_{i,j},u_{i,j}^{(\ell)},d_i)+\sum_{i=1}^r\sum_{j=1}^{q_i} \Big(\eta_{i,j}-\sum_{\ell=1}^{k_{i,j}}\eta_{i,j}^{(\ell)}\Big)(c_i,\beta_{i,j},0,d_i).
\end{align*}
\quad Therefore, $(a,b,n)\in\mathbb{N}\mathcal{H}_{\mathcal{P}_w}$, and the induction process is complete. 

To show the inclusion ``$\supset$'' take $(c_i,b,d_i)\in\mathcal{H}_{\mathcal{P}_w}$, $b=(b_1,\ldots,b_p)$. 
It suffices to show that $(c_i,b,d_i)\in\mathbb{R}_+\Gamma$. 
As $c_i\in d_i\mathcal{P}$ and $0\leq b_j\leq w_j(c_i)$ for $j=1,\ldots,p$, one has
\begin{align*}
c_i&=d_i\sum_{\ell=1}^m\epsilon_\ell v_\ell,\ \epsilon_\ell\geq 0,\ \sum_{\ell=1}^m\epsilon_\ell=1,\\
0&\leq b_j\leq w_j(c_i)=d_i\sum_{\ell=1}^m\epsilon_\ell w_j(v_\ell),\ j=1,\ldots,p.
\end{align*}
\quad Hence, noticing that $b_i$ is the $(s+i)$-entry of $(c_i,b,d_i)$ and using Lemma~\ref{sep15-22}, we obtain
\begin{align*}
b_j&=d_i\sum_{\ell=1}^m\epsilon_\ell\mu_{j,\ell},\ 0\leq\mu_{j,\ell}\leq w_j(v_\ell),\,\ell=1,\ldots,m,\ j=1,\ldots,p,\\
(c_i,b,d_i)&=d_i\Big[\sum_{\ell=1}^m\epsilon_\ell(v_\ell,0)+\sum_{j=1}^p\Big(\sum_{\ell=1}^m\epsilon_\ell\mu_{j,\ell}\Big)e_{s+j}+\Big(\sum_{\ell=1}^m\epsilon_\ell\Big)e_{s+p+1}\Big]\\
&=d_i\Big[\sum_{\ell=1}^m\epsilon_\ell\Big((v_\ell,0)+\Big(\sum_{j=1}^p\mu_{j,\ell}e_{s+j}\Big)+e_{s+p+1}\Big)\Big].
\end{align*}
\quad Thus, setting $\gamma_\ell:=(v_\ell,0)+(\sum_{j=1}^p\mu_{j,\ell}e_{s+j})+e_{s+p+1}$, we need only show that $\gamma_\ell\in\mathbb{R}_+\Gamma$ for $\ell=1,\ldots,m$. 
Recall that 
$$
[0,1]^p={\rm conv}(\{0\}\cup\{e_{j_1}+\cdots+e_{j_k}\mid 1\leq j_1<\cdots<j_k\leq p\}).
$$
\quad Fix $1\leq\ell\leq m$ and consider the linear map $T_\ell\colon\mathbb{R}^p\rightarrow\mathbb{R}^p$, $e_i\mapsto w_i(v_\ell)e_i$, $i=1,\ldots,p$. 
$T_\ell$ is linear and commutes with taking convex hulls, by Lemma~\ref{image-cube}, we get
\begin{align*}
&[0,w_1(v_\ell)]\times\cdots\times[0,w_p(v_\ell)]=T_\ell([0,1]^p)\\
&=T_\ell({\rm conv}(\{0\}\cup\{e_{j_1}+\cdots+e_{j_k}\mid 1\leq j_1<\cdots<j_k\leq p\}))\\
&\quad\quad ={\rm conv}(\{0\}\cup\{w_{j_1}(v_\ell)e_{j_1}+\cdots+w_{j_k}(v_\ell)e_{j_k}\mid 1\leq j_1<\cdots<j_k\leq p\}).
\end{align*}  
\quad Hence, as $0\leq\mu_{j,\ell}\leq w_j(v_\ell)$ for $j=1,\ldots,p$, one obtains
\begin{align*}
&\sum_{j=1}^p\mu_{j,\ell}e_j\in{\rm conv}(\{0\}\cup\{w_{j_1}(v_\ell)e_{j_1}+\cdots+w_{j_k}(v_\ell)e_{j_k}\mid 1\leq j_1<\cdots<j_k\leq p\}),\\
&\sum_{j=1}^p\mu_{j,\ell}e_j=\lambda_1\cdot0+\sum\lambda_{j_1,\ldots,j_k} w_{j_1}(v_\ell)e_{j_1}+\cdots+w_{j_k}(v_\ell)e_{j_k},\\
&\lambda_1\geq 0,\ \lambda_{j_1,\ldots,j_k}\geq 0,\ \lambda_1+\sum\lambda_{j_1,\ldots,j_k}=1,\\
&(v_\ell,0)+\sum_{j=1}^p\mu_{j,\ell}e_{s+j}=(\lambda_1+\sum\lambda_{j_1,\ldots,j_k})(v_\ell,0)+\lambda_1\cdot0\\
&\quad +\sum\lambda_{j_1,\ldots,j_k} w_{j_1}(v_\ell)e_{s+j_1}+\cdots+w_{j_k}(v_\ell)e_{s+j_k}\\
&=\lambda_1(v_\ell,0)+\sum\lambda_{j_1,\ldots,j_k}\big( (v_\ell,0)+w_{j_1}(v_\ell)e_{s+j_1}+\cdots+w_{j_k}(v_\ell)e_{s+j_k}\big).
\end{align*}
\quad Then, adding $e_{s+p+1}=(\lambda_1+\sum\lambda_{j_1,\ldots,j_k})e_{s+p+1}$ to the last two equalities, we get
\begin{align*}
&\gamma_\ell=(v_\ell,0)+\Big(\sum_{j=1}^p\mu_{j,\ell}e_{s+j}\Big)+e_{s+p+1}=\lambda_1((v_\ell,0)+e_{s+p+1})\\
&\quad\quad+\sum\lambda_{j_1,\ldots,j_k}\big( (v_\ell,0)+w_{j_1}(v_\ell)e_{s+j_1}+\cdots+w_{j_k}(v_\ell)e_{s+j_k}+e_{s+p+1}\big).
\end{align*}
\quad Thus, $\gamma_\ell\in\mathbb{R}_+\Gamma$ for $\ell=1,\ldots,m$.

(b) The semigroup ring $K[\mathbb{N}\mathcal{H}_{\mathcal{P}_w}]$ is equal to the monomial subring 
$$
K[t^aq^bz^n\mid (a,b,n)\in\mathbb{N}\mathcal{H}_{\mathcal{P}_w}],
$$
and the inclusion $A_r^w(\mathcal{P})\supset K[\mathbb{N}\mathcal{H}_{\mathcal{P}_w}]$ follows noticing that $t^{c_i}q^bz^{d_i}$ is in $A_r^w(\mathcal{P})$ for all $(c_i,b,d_i)$ in $\mathcal{H}_{\mathcal{P}_w}$. 
The inclusion $A_r^w(\mathcal{P})\subset K[\mathbb{N}\mathcal{H}_{\mathcal{P}_w}]$ follows from the proof of part (a). 
For clarity, we briefly explain the main steps. 
We proceed by induction on $p\geq 1$. 
Assume that $p=1$. 
Take $t^aq_1^bz^n\in A_r^w(\mathcal{P})$.
Then, $a\in n\mathcal{P}\cap\mathbb{Z}^s$, $b\in\mathbb{N}$, and $0\leq b\leq w_1(a)$. 
Thus, $t^az^n\in A(\mathcal{P})$ and consequently, by Eq.~\eqref{jan20-23}, we get 
\begin{align*}
t^az^n&=(t^{c_1}z^{d_1})^{\eta_1}\cdots (t^{c_r}z^{d_r})^{\eta_r},\quad \eta_i\in\mathbb{N},\\
(a,n)&=\eta_1(c_1,d_1)+\cdots+\eta_r(c_r,d_r),\\
b&\leq w_1(a)=\eta_1w_1(c_1)+\cdots+\eta_rw_1(c_r).
\end{align*}
\quad We can now proceed as in the case $p=1$ of part (a), starting with Eq.~\eqref{touselater}, to obtain that $(a,b,n)\in\mathbb{N}\mathcal{H}_{\mathcal{P}_w}$, that is, $t^aq_1^bz^n\in K[\mathbb{N}\mathcal{H}_{\mathcal{P}_w}]$. 
Assume that $p>1$. 
Take $t^aq^bz^n\in A_r^w(\mathcal{P})$, $b=(b_1,\ldots,b_p)$. 
In particular one has $0\leq b_i\leq w_i(a)$ for $i=1,\ldots,p$. 
With the notation of Eq.~\eqref{oct29-22-1}, by induction we get 
$$
t^aq_1^{b_1}\cdots q_{p-1}^{b_{p-1}}z^n\in K[\mathbb{N}\mathcal{H}_{\mathcal{P}'}],
$$
that is, setting $b'=(b_1,\ldots,b_{p-1})$, with the notation of Eqs.~\eqref{oct29-22-4}--\eqref{oct29-22-5}, one has 
\begin{equation*}
(a,b',n)=\sum_{j=1}^{q_1}\eta_{1,j}(c_1,\beta_{1,j},d_1)+\cdots+
\sum_{j=1}^{q_r}\eta_{r,j}(c_r,\beta_{r,j},d_r).
\end{equation*}
\quad As $0\leq b_p\leq w_p(a)$, starting with Eq.~\eqref{oct26-22}, we can proceed as in the case $p>1$
of part (a) to obtain that $(a,b,n)\in\mathbb{N}\mathcal{H}_{\mathcal{P}_w}$.

(c) As ${\mathcal{P}_w}={\rm conv}(\mathcal{G})$ and $\Gamma=\{(\alpha,1)\mid \alpha\in\mathcal{G}\}$, by
\cite[Theorem~9.3.6(b)]{monalg-rev}, one has
$$
A({\mathcal{P}_w})=K[t^aq^bz^n\mid (a,b,n)\in\mathbb{R}_+\Gamma\cap\mathbb{Z}^{s+p+1}]=K[\mathbb{R}_+\Gamma\cap\mathbb{Z}^{s+p+1}].
$$
\quad Then, by parts (a) and (b), $A({\mathcal{P}_w})=K[\mathbb{N}\mathcal{H}_{\mathcal{P}_w}]=A_r^w(\mathcal{P})$. 

(d) By part (c), $A({\mathcal{P}_w})=A_r^w(\mathcal{P})$. 
Thus, by Lemma~\ref{jan7-23}, we get
$$
E_{\mathcal{P}_w}(n)=\dim_K(A({\mathcal{P}_w})_n)=\dim_K(A_r^w(\mathcal{P})_n)=E_\mathcal{P}^{s, \scriptstyle\prod_{i=1}^p(w_i+1)}(n),
$$
and the proof of (d) is complete.

(e) By part (a), $(a,b,n)\in
\mathbb{R}_+\Gamma\cap\mathbb{Z}^{s+p+1}$ if and only if
$(a,b,n)\in\mathbb{N}\mathcal{H}_{\mathcal{P}_w}$ and, by part (b), 
$(a,b,n)\in\mathbb{N}\mathcal{H}_{\mathcal{P}_w}$ if and only if 
$(a,n)\in
{\rm cone}(\mathcal{P})\cap\mathbb{Z}^{s+1}$ and $0\leq b_i\leq w_i(a)$ for all
$i$, where $b=(b_1,\ldots,b_p)$. 
Then, by
\cite[Corollary~3.7]{BeckRobins}, the integer-point transform of 
the cone $\mathbb{R}_+\Gamma$ is a rational function given by
\begin{align*}
\sum_{\scriptstyle (a,b,n)\in\mathbb{R}_+\Gamma\cap\, \ZZ^{s+p+1}} \hspace{-7mm} t^{a}q^bx^n 
=\hspace{-7mm}\sum_{\begin{array}{c}\scriptstyle (a,n)\in{\rm cone}(\mathcal{P})\cap \ZZ^{s+1}\vspace{-1mm}\\ 
{\scriptstyle 0 \leq b_i\leq w_i(a)\ \forall\, i}\end{array}} \hspace{-7mm} t^{a}q^bx^n=F_\mathcal{P}^{r,w}(t,q,x).
\end{align*}
\quad Thus, $F_\mathcal{P}^{r,w}(t,q,x)$ is a rational function. 
\end{proof}

\begin{remark} 
From Theorem~\ref{weighted-ehrhart}(c) it follows that $t^aq^bz^n\in A_r^w(\mathcal{P})_n$ if and only if $(a,b)\in n{\mathcal{P}_w}$.
\end{remark}

The following result says in particular that if $\mathcal{P}$ is
full-dimensional, that is, $\dim(\mathcal{P})=s$ and $w$ is a
monomial of degree $p$, then the degree of $E_{\mathcal{P}_w}$ is
$\dim(\mathcal{P})+p$.   

\begin{proposition}\label{nov6-22} Let $w=w_1\cdots w_p$ be any monomial of $S$ of degree $p$ with $w_i\in\{t_1,\ldots,t_s\}$ for $i=1,\ldots,p$. 
If $w_i\not\equiv 0$ on $\mathcal{P}$ for $i\leq r$ and $w_i\equiv 0$ on $\mathcal{P}$ for $i>r$, then
$$
\dim({\mathcal{P}_w})=\deg(E_{\mathcal{P}_w})=\dim(\mathcal{P})+r.
$$
\end{proposition}
\begin{proof} The equality $\dim({\mathcal{P}_w})=\deg(E_{\mathcal{P}_w})$ is well known, see \cite{BeckRobins}, \cite[p.~383]{monalg-rev}. 
Let $D$ be the following matrix of size $(m\sum_{k=0}^p\binom{p}{k})\times(s+p+1)$ whose rows correspond to the vectors in $\Gamma$
\begin{equation*}
D=
\left[\begin{matrix} 
v_1&0&\cdots&0&1\\ 
\vdots& &&&\vdots\\ 
v_m&0&\cdots&0&1\\ 
v_i& &\sum_{\ell=1}^kw_{j_\ell}(v_i)e_{s+j_\ell}
&&1
\end{matrix}
\right]_{\scriptstyle 1\leq i\leq m,\, 1\leq j_1<\cdots<j_k\leq p}
\end{equation*}
\quad By using elementary row operations, $D$ is equivalent to 
\begin{equation*}
D_1=
\left[\begin{matrix} 
v_1&0&\cdots&0&1\\ 
\vdots& &&&\vdots\\ 
v_m&0&\cdots&0&1\\ 
0& &\sum_{\ell=1}^kw_{j_\ell}(v_i)e_{s+j_\ell}
&&0
\end{matrix}
\right]_{\scriptstyle 1\leq i\leq m,\, 1\leq j_1<\cdots<j_k\leq p}
\end{equation*}
\quad Pick $v_{\ell_1},\ldots v_{\ell_r}$ in $\{v_1,\ldots,v_m\}$ such that $w_1(v_{\ell_1})\neq 0,\ldots,w_r(v_{\ell_r})\neq 0$. 
The matrix
\begin{equation*}
D_2=
\left[\begin{matrix} 
v_1&0&\cdots&0&1\\ 
\vdots& &&&\vdots\\ 
v_m&0&\cdots&0&1
\end{matrix}\right]
\end{equation*}
has rank $\dim(\mathcal{P})+1$ (see the proof of \cite[Proposition~9.3.1]{monalg-rev}) whereas the matrix
\begin{equation*}
D_3=
\left[\begin{matrix} 
0& &\sum_{\ell=1}^kw_{j_\ell}(v_i)e_{s+j_\ell}
&&0
\end{matrix}
\right]_{\scriptstyle 1\leq i\leq m,\, 1\leq j_1<\cdots<j_k\leq p}
\end{equation*}
has rank at most $r$ because $w_i=0$ on $\mathcal{P}$ for $i>r$. 
To show that the rank of $D_3$ is $r$ notice that $w_1(v_{\ell_1})e_{s+1},\ldots,w_r(v_{\ell_r})e_{s+r}$ are linearly independent rows of $D_3$. 
Therefore
\begin{align*}
\dim({\mathcal{P}_w})+1&=\dim(A({\mathcal{P}_w}))={\rm rank}(D)\\
&={\rm rank}(D_1)={\rm rank}(D_2)+{\rm rank}(D_3)=\dim(\mathcal{P})+1+r,
\end{align*}
where the first two equalities follow from \cite[Theorem~9.3.6(b), Corollary~8.2.21]{monalg-rev}. 
Hence $\dim({\mathcal{P}_w})$ is equal to $\dim(\mathcal{P})+r$ and the proof is complete.
\end{proof}

We end with an example that illustrates the results of this section:  

\begin{example}\label{example0} First, we give an illustration of the equality $E_{\mathcal{P}_w}=E_\mathcal{P}^{s, \scriptstyle\prod_{i=1}^p(w_i+1)}$ of Theorem~\ref{weighted-ehrhart}.
Let $\mathcal{P}={\rm conv}(v_1,v_2,v_3,v_4)$, $v_1=(0,0)$, $v_2=(1,0)$, $v_3=(0,1)$, $v_4=(1,1)$, be the unit square in $\mathbb{R}^2$ and let $w_i$, $i=1,2,3$, be the linear functions $w_1=t_1+t_2$, $w_2=2t_1+3t_2$ and $w_3=t_1$. 
Using \textit{Macaulay}$2$ \cite{mac2}, we obtain 	
\begin{align*}
E_\mathcal{P}(n)&=(n+1)^2,\\
E_\mathcal{P}^{s, w_1}(n)&=(n^3+2n^2+n),\\
E_\mathcal{P}^{s, w_2}(n)&=(5/2)n^3+5n^2+(5/2)n,\\
E_\mathcal{P}^{s, w_3}(n)&=(1/2)n^3+n^2+(1/2)n,\\
E_\mathcal{P}^{s, w_1w_2}(n)&=(35/12)n^4+(20/3)n^3+(55/12)n^2+(5/6)n,\\
E_\mathcal{P}^{s, w_1w_3}(n)&=(17/12)n^4+(19/6)n^3+(25/12)n^2+(1/3)n,\\
E_\mathcal{P}^{s, w_2w_3}(n)&=(7/12)n^4+(4/3)n^3+(11/12)n^2+(1/6)n,\\
E_\mathcal{P}^{s, w_1w_2w_3}(n)&=(11/6)n^5+(29/6)n^4+(25/6)n^3+(7/6)n^2.
\end{align*}
Therefore, by Theorem~\ref{weighted-ehrhart}, we get
\begin{align*}
E_{\mathcal{P}_w}(n)=&E_\mathcal{P}^{s, \scriptstyle\prod_{i=1}^p(w_i+1)}(n)= E_\mathcal{P}(n)+E_\mathcal{P}^{s, w_1}(n)+E_\mathcal{P}^{s, w_2}(n)+E_\mathcal{P}^{s, w_3}(n)+\\
&E_\mathcal{P}^{s, w_1w_2}(n)+E_\mathcal{P}^{s, w_1w_3}(n)+E_\mathcal{P}^{s, w_2w_3}(n)+E_\mathcal{P}^{s, w_1w_2w_3}(n)\\
=&(11/6)n^5+(39/4)n^4+(58/3)n^3+(71/4)n^2+(22/3)n+1,
\end{align*}
where ${\mathcal{P}_w}={\rm conv}(\mathcal{G})\subset \mathbb{R}^5$ is the $r$-weight lifting polytope and $\mathcal{G}$ is the multiset of vectors of the form $(v_i,0)$ and all $(v_i,\sum_{\ell\in I} w_\ell(v_i)e_{s+\ell})$ for $I\subseteq[p]$, see Eq.~\eqref{dec11-22-1}. We implemented the construction in \textit{Macaulay}$2$ \cite{mac2}. Computing $w_i(v_j)$ for all $i,j$, this list of vectors becomes
\begin{verbatim}
     (0, 0, 0, 0, 0),  (0, 0, 0, 0, 0),  (1, 0, 0, 0, 0),  (1, 0, 1, 2, 0),  
     (0, 1, 0, 0, 0),  (0, 1, 1, 3, 0),  (1, 1, 0, 0, 0),  (1, 1, 2, 5, 0),  
     (0, 0, 0, 0, 0),  (0, 0, 0, 0, 0),  (1, 0, 1, 0, 0),  (1, 0, 1, 0, 1), 
     (0, 1, 1, 0, 0),  (0, 1, 1, 0, 0),  (1, 1, 2, 0, 0),  (1, 1, 2, 0, 1),  
     (0, 0, 0, 0, 0),  (0, 0, 0, 0, 0),  (1, 0, 0, 2, 0),  (1, 0, 0, 2, 1),  
     (0, 1, 0, 3, 0),  (0, 1, 0, 3, 0),  (1, 1, 0, 5, 0),  (1, 1, 0, 5, 1), 
     (0, 0, 0, 0, 0),  (0, 0, 0, 0, 0),  (1, 0, 0, 0, 1),  (1, 0, 1, 2, 1),  
     (0, 1, 0, 0, 0),  (0, 1, 1, 3, 0),  (1, 1, 0, 0, 1),  (1, 1, 2, 5, 1).
\end{verbatim}
\quad The Ehrhart series of $\mathcal{P}$ and ${\mathcal{P}_w}$ are
$$
F_{\mathcal{P}}(x)=\frac{1+ x}{(1-x)^3}\ \mbox{ and }\ F^{s,w}_{\mathcal{P}}(x) = F_{\mathcal{P}_w}(x)=\frac{1+ 51x+ 129x^2+ 39x^3}{(1-x)^6},\ \mbox{ respectively}.
$$
\end{example}

\section{$s$-weighted Ehrhart functions and Ehrhart series}\label{s-weighted-section}

In this section, we show applications to $s$-weighted Ehrhart theory. 
To avoid repetitions, we continue to employ the notations and definitions used in Section~\ref{r-weighted-section}.  

\begin{proposition}\label{nov8-22} 
Let $w=w_1\cdots w_p$ be any monomial of $S$ of degree $p$ with $w_i\in\{t_1,\ldots,t_s\}$ for $i=1,\ldots,p$. 
If $w_i\not\equiv 0$ on $\mathcal{P}$ for $i=1,\ldots,p$, then $E_\mathcal{P}^{s,w}$ is a polynomial whose leading coefficient is the relative volume ${\rm vol}({\mathcal{P}_w})$ of ${\mathcal{P}_w}$ and 
$$
\deg(E^{s, w}_\mathcal{P})=\dim(\mathcal{P})+p.
$$
\end{proposition}
\begin{proof} 
We argue by induction on $p\geq 1$. 
By Theorem~\ref{weighted-ehrhart}, we obtain
\begin{align}\label{nov8-22-1}
E_{\mathcal{P}_w}(n)&=E_\mathcal{P}^{s, w_1\cdots w_p}(n)+\cdots+\Big({\sum_{1\leq j_1<\cdots< j_k\leq p}}E_\mathcal{P}^{s, w_{j_1}\cdots w_{j_k}}(n)\Big)+\cdots+E_\mathcal{P}(n).
\end{align}
\quad Assume that $p=1$. 
Then, by Eq.~\eqref{nov8-22-1}, we obtain
\begin{equation}\label{jan22-23}
E_{\mathcal{P}_w}(n)=E_\mathcal{P}^{s, w}(n)+E_\mathcal{P}(n)\mbox{ for }n\geq 0.
\end{equation}
\quad 
Since $E_{\mathcal{P}_w}$ and $E_\mathcal{P}$ are Ehrhart functions
of lattice polytopes, one has
$\deg(E_{\mathcal{P}_w})=\dim({\mathcal{P}_w})$,
$\deg(E_\mathcal{P})=\dim(\mathcal{P})$, the leading coefficient of
$E_{\mathcal{P}_w}$ is ${\rm vol}({\mathcal{P}_w})$ and, by
Proposition~\ref{nov6-22}, we get that $\deg(E_{\mathcal{P}_w})=\dim(\mathcal{P})+1$. 
Hence, by Eq.~\eqref{jan22-23}, $E_\mathcal{P}^{s, w}$ is a polynomial of degree $\dim(\mathcal{P})+1$ whose leading coefficient is ${\rm vol}({\mathcal{P}_w})$. 
Assume that $p>1$. 
By induction, using Eq.~\eqref{nov8-22-1}, we obtain
$$
E_{\mathcal{P}_w}(n)=E_\mathcal{P}^{s, w}(n)+g(n)
\mbox{ for }n\geq0,
$$
where $g$ is a polynomial of degree $\dim(\mathcal{P})+p-1$. 
Thus, by Proposition~\ref{nov6-22}, $E_\mathcal{P}^{s, w}=E_{\mathcal{P}_w}-g$ is a polynomial of degree $\dim(\mathcal{P})+p$ whose leading coefficient is ${\rm vol}({\mathcal{P}_w})$.
\end{proof}

A function $E\colon\mathbb{N}\rightarrow K$ is called a \textit{polynomial} if there is a polynomial $g(x)\in K[x]$ such that $E(n)=g(n)$ for all $n\in\mathbb{N}$. 
If $E(n)=g(n)$ for $n\gg 0$, we say that $E$ is a \textit{polynomial function}. 

The Ehrhart function $E_\mathcal{P}$ of a lattice polytope $\mathcal{P}$ is a polynomial in $n$ and so is the weighted Ehrhart function $E_\mathcal{P}^{s, f}$, $f\in S$.

\begin{corollary}\label{nonneg-poly} 
Let $f$ be a non-zero polynomial of $S=K[t_1,\ldots,t_s]$ of degree $p$, $\mathbb{R}\subset K$, and let $d$ be the dimension of the lattice polytope $\mathcal{P}$. 
The following hold.
\begin{enumerate}
\item[(a)] $E_\mathcal{P}^{s, f}$ is a polynomial of $n$ of degree at most $d+p$ and $F_\mathcal{P}^{s,f}(x)$ is a rational function.
\item[(b)] If $\mathcal{P}$ is non-degenerate and $f$ is a monomial, then $\deg(E_\mathcal{P}^{s, f})=d+p$. 
\item[(c)] $E_\mathcal{P}^{s, f}$ is a $K$-linear combination of Ehrhart polynomials. 
\item[(d)]\cite[Proposition~4.1]{Brion-Vergne} If the interior $\mathcal{P}^{\rm o}$ of $\mathcal{P}$ is nonempty, $K=\mathbb{R}$, and $f$ is homogeneous and $f\geq 0$ on $\mathcal{P}$, then $E_\mathcal{P}^{s, f}$ is a polynomial of degree $s+p$.    
\end{enumerate}
\end{corollary}
\begin{proof} 
We can write $f=\sum_{b}\mu_bt^b$, $\mu_b\in K\setminus\{0\}$. 
Then, for $n\geq 0$, one has 
\begin{align}\label{nov28-22}
E_\mathcal{P}^{s, f}(n)=\hspace{-3mm}\sum_{a\in n\mathcal{P}\cap\mathbb{Z}^s}\hspace{-3mm}f(a)=\sum_{a\in n\mathcal{P}\cap\mathbb{Z}^s}\Big(\sum_{b}\mu_bt^b(a)\Big)= \sum_{b}\Big(\mu_b\hspace{-3mm}\sum_{a\in n\mathcal{P}\cap\mathbb{Z}^s}t^b(a)\Big)=\sum_{b}\mu_b E_\mathcal{P}^{s, t^b}(n).       
\end{align}

(a) By Proposition~\ref{nov8-22}, it follows readily that $E_\mathcal{P}^{s, t^b}$ is a polynomial of $n$ of degree at most $d+p$ for all $b$.  
Thus, by Eq.~\eqref{nov28-22}, $E_\mathcal{P}^{s, f}$ is a polynomial of $n$ of degree at most $d+p$. Then, by \cite[Corollary~4.3.1]{Sta5}, $F_\mathcal{P}^{s, f}(x)$ is a rational function. 

(b) Writing $f=w_1\cdots w_p$ with $w_i\in\{t_1,\ldots,t_s\}$ for all $i$ and noticing that $\mathcal{P}$ is non-degenerate if and only if $t_i\not\equiv 0$ on $\mathcal{P}$ for $i=1,\ldots,s$, by Proposition~\ref{nov8-22}, we get $\deg(E_\mathcal{P}^{s, f})=d+p$.

(c) By Eq.~\eqref{nov28-22}, it suffices to show that $E_\mathcal{P}^{s, t^b}$ is a $K$-linear combination of Ehrhart polynomials, and this follows using induction on $\deg(t^b)$ and Theorem~\ref{weighted-ehrhart}(d).

(d) As $\mathcal{P}^{\rm o}\neq\emptyset$, the dimension of $\mathcal{P}$ is equal to $s$, and consequently $t_i\not\equiv 0$ on $\mathcal{P}$ for $i=1,\ldots,s$. 
Then, by Proposition~\ref{nov8-22}, $E_\mathcal{P}^{s, t^b}$ is a polynomial of degree $s+p$ for all $t^b$. Let $u_b$ be the leading coefficient of $E_\mathcal{P}^{s, t^b}$. 
By Eq.~\eqref{nov28-22}, it suffices to show that $\sum_b\mu_bu_b\neq 0$. 
Setting $h=t^b$ and noticing that $h(a)=n^ph(\frac{a}{n})$ for all $a\in\mathbb{R}^s$ and $n\geq 1$, one has
\begin{align*}
u_b=\lim_{n\rightarrow\infty}\frac{E_\mathcal{P}^{s, h}(n)}{n^{s+p}}=\lim_{n\rightarrow\infty}\sum_{a\in n\mathcal{P}\cap\mathbb{Z}^s}\frac{h(\frac{a}{n})}{n^{s}}=\lim_{n\rightarrow\infty}\sum_{\frac{a}{n}\in \mathcal{P}\cap(\frac{1}{n}\mathbb{Z}^s)}\frac{h(\frac{a}{n})}{n^{s}}=\lim_{n\rightarrow\infty}\sum_{c\in \mathcal{P}\cap(\frac{1}{n}\mathbb{Z}^s)}\frac{h(c)}{n^{s}}=\int_\mathcal{P}h,
\end{align*}
where the last equality is followed by elementary integration theory. 
Thus, $u_b=\int_\mathcal{P}t^b$. The fact that the leading coefficient of $E_\mathcal{P}^{s, h}$ is equal to $\int_\mathcal{P}h$ appears in \cite[p.~437]{baldoni-etal}, and more recently in \cite[Proposition~5]{Bruns-Soger}. 
Therefore      
\begin{equation}\label{nov28-22-1}
\sum_b\mu_bu_b=\sum_b \mu_b\int_\mathcal{P}t^b=\int_\mathcal{P}\sum_b\mu_bt^b=\int_\mathcal{P}f.
\end{equation}
\quad As $\mathcal{P}^{\rm o}\neq\emptyset$, there is a box $B=B_1\times\cdots\times B_s$ contained in an open subset of $\mathcal{P}$, where $B_i$ is a closed interval of $\mathbb{R}$ with $|B_i|=\infty$ for all $i$. 
Then, by Alon's combinatorial Nullstellensatz \cite{alon-cn}, \cite[Theorem~8.4.11]{monalg-rev}, there is $\beta\in B$ with $f(\beta)>0$. 
Therefore, $\int_\mathcal{P}f>0$ because $f$ is a continuous function and $f\geq 0$ on $\mathcal{P}$. 
Hence, by Eq.~\eqref{nov28-22-1}, $\sum_b\mu_bu_b>0$.  
\end{proof}

\begin{proposition}\label{ineq-ehrhart}
Let $f$ be a homogeneous polynomial of degree $p$ in $K[t_1,\ldots,t_s]$, $K=\mathbb{R}$, and let $\mathcal{P}$ be a lattice polytope in $\mathbb{R}^s$. 
Then, there are $c_1,c_2\in\mathbb{R}$, $c_1\leq c_2$, such that
$$
c_1n^pE_\mathcal{P}(n)\leq E_\mathcal{P}^{s, f}(n)\leq c_2n^pE_\mathcal{P}(n)\mbox{ for }n\geq 0.
$$
\end{proposition}
\begin{proof} 
For $a\in n\mathcal{P}\cap\mathbb{Z}^s$ and $n\geq 1$, we can write $a=nb$, with $b\in\mathcal{P}$. 
Hence, $f(a)=n^pf(b)$ because $f$ is homogeneous of degree $p$. 
We set 
$$
c_1=\inf\{f(x)\mid x\in\mathcal{P}\}\mbox{ and }c_2=\sup\{f(x)\mid x\in\mathcal{P}\}.
$$
\quad Since $f$ is a continuous function and $\mathcal{P}$ is a compact subset of $\mathbb{R}^s$, $c_1$ and $c_2$ are real numbers and $f(\mathcal{P})\subset[c_1,c_2]$. 
Thus  $n^pc_1\leq f(a)=n^pf(b)\leq n^pc_2$, therefore $c_1n^pE_\mathcal{P}(n)\leq E_\mathcal{P}^{s, f}(n)=\sum_{\scriptstyle a\in n\mathcal{P}\cap\mathbb{Z}^s}f(a)\leq c_2n^pE_\mathcal{P}(n),$
and the proof is complete. 
\end{proof}

\begin{corollary}\label{linear-weighted-ehrhart} Let $v_1,\ldots,v_m$ be points in $\mathbb{N}^s$, let $\mathcal{P}={\rm conv}(v_1,\ldots,v_m)$, and let $w\colon\mathbb{R}^s\rightarrow\mathbb{R}$ be a non-zero linear function such that $w(e_i)\in\mathbb{N}$ for $i=1,\ldots,s$. 
The following hold.     
$$
\leqno\ \ \ {\rm(a)}\quad E_\mathcal{P}^{s,w}(n):=\sum_{a\in n\mathcal{P}\cap\mathbb{Z}^s}w(a)=E_{\mathcal{P}_{w}}(n)-E_\mathcal{P}(n)\mbox{ for }n\geq 0,
$$
where ${\mathcal{P}_w}={\rm conv}((v_1,0),\ldots,(v_m,0),(v_1,w(v_1)),\ldots(v_m,w(v_m)))$.

{\rm (b)} If $w=\sum_{i=1}^s\eta_it_i$, $t_i\not\equiv 0$ on $\mathcal{P}$ whenever $\eta_i>0$, and $d=\dim(\mathcal{P})$, then $E_\mathcal{P}^{s, w}$ is a polynomial of degree $d+1$ and the generating function
$F_\mathcal{P}^{s, w}$ of $E_\mathcal{P}^{s, w}$ is equal to   
$$
F_\mathcal{P}^{s, w}(x):=\sum_{n=0}^\infty E_\mathcal{P}^{s, w}(n)x^n=\frac{h(x)}{(1-x)^{d+2}},
$$
where $h(x)$ is a polynomial of degree at most $d+1$ with nonnegative integer coefficients.
\end{corollary}
\begin{proof} 
(a) From Theorem~\ref{weighted-ehrhart}(d), we get $E_{\mathcal{P}_{w}}(n)=E_\mathcal{P}^{s, w+1}(n)=E_\mathcal{P}^{s, w}(n)+E_\mathcal{P}(n)$ for $n\geq 0$.

(b) If $\eta_i>0$, by Proposition~\ref{nov8-22}, $\eta_iE_\mathcal{P}^{s, t_i}$ is a polynomial of degree $\dim(\mathcal{P})+1$ whose leading coefficient is positive. 
As $E_\mathcal{P}^{s, w}=\sum_{i=1}^s\eta_iE_\mathcal{P}^{s, t_i}$, $E_\mathcal{P}^{s, w}$ is a polynomial of degree $\dim(\mathcal{P})+1$. 
By part (a), Proposition~\ref{nov6-22}, and Stanley's result that $h^*$-vectors of lattice polytopes have nonnegative integer coefficients \cite[Theorem~2.1]{Stanley-nonneg-h-vector}, one has
\begin{align*} 
&F_\mathcal{P}^{s, w}(x)=F_{\mathcal{P}_w}(x)-F_\mathcal{P}(x)=\frac{f(x)}{(1-x)^{d+2}}-\frac{(1-x)g(x)}{(1-x)^{d+2}}=\frac{h(x)}{(1-x)^{d+2}},
\end{align*}
where $f(x)$ and $g(x)$ are polynomials with non-negative integer coefficients of degrees at most $d+1$ and $d$, respectively. 
Note that $\mathcal{P}':={\rm conv}((v_1,0),\ldots,(v_m,0))\subset{\mathcal{P}_w}$, the Ehrhart series of $\mathcal{P}'$ and $\mathcal{P}$ are equal, and $f(x)$ and $g(x)$ are the $h^*$-polynomials of ${\mathcal{P}_w}$ and $\mathcal{P}$, respectively. 
Then, by Stanley's  monotonicity property of $h^*$-vectors \cite[Theorem~3.3]{stanley-mono} (cf. \cite[Theorem~3.3]{beck-deco}), we get that $h^*({\mathcal{P}_w})\geq h^*(\mathcal{P}')=h^*(\mathcal{P})$ componentwise, and consequently $f(x)\geq g(x)\geq 0$ coefficientwise. 
Therefore, from the equalities
$$
h(x)=f(x)-(1-x)g(x)=f(x)-g(x)+xg(x),
$$
we obtain that $h(x)$ has nonnegative integer coefficients.
\end{proof}

We conclude this section illustrating Theorem \ref{jdl} and
Corollary~\ref{linear-weighted-ehrhart}.
\begin{example}\label{hope-figure} 
Let $\mathcal{P}=[0,1]\times [0,1]$ be the unit square in
$\mathbb{R}^2$ and let $w$ be the linear weight $w(y_1,y_2)=y_1+y_2$.
The weight lifting polytopes $\mathcal{P}^w$ and $\mathcal{P}_w$ of
$\mathcal{P}$ relative to $w$ are given by  
\begin{align*}
\mathcal{P}^w=&{\rm conv}((0,0,0),(1,0,1),(0,1,1),(1,1,2)),\\
\mathcal{P}_w&={\rm conv}((0,0,0),(1,0,0),(0,1,0),(1,1,0),(1,0,1),(0,1,1),(1,1,2)).
\end{align*}
\quad The polytope $\mathcal{P}_w$ is depicted in Figure~\ref{figure-vila}. 
The ``upper facet'' of $\mathcal{P}_w$ is precisely
$\mathcal{P}^w$ and $\mathcal{P}_w$ contains not only $\mathcal{P}^w$
but also all points below $\mathcal{P}^w$.   
By, Theorem \ref{jdl} and Corollary~\ref{linear-weighted-ehrhart}, one has  
$$
E_\mathcal{P}^{s,w}(n)=\sum_{a\in n\mathcal{P}\cap\mathbb{Z}^2}w(a)=
E_{\mathcal{P}_w}(n)-E_{\mathcal{P}^w}(n)=E_{\mathcal{P}_w}(n)-E_{\mathcal{P}}(n)\mbox{
for all }n\geq 0.  
\vspace{-0.6cm}
$$
\begin{figure}[htp]
\includegraphics[scale=0.6,height=2in,width=2in]{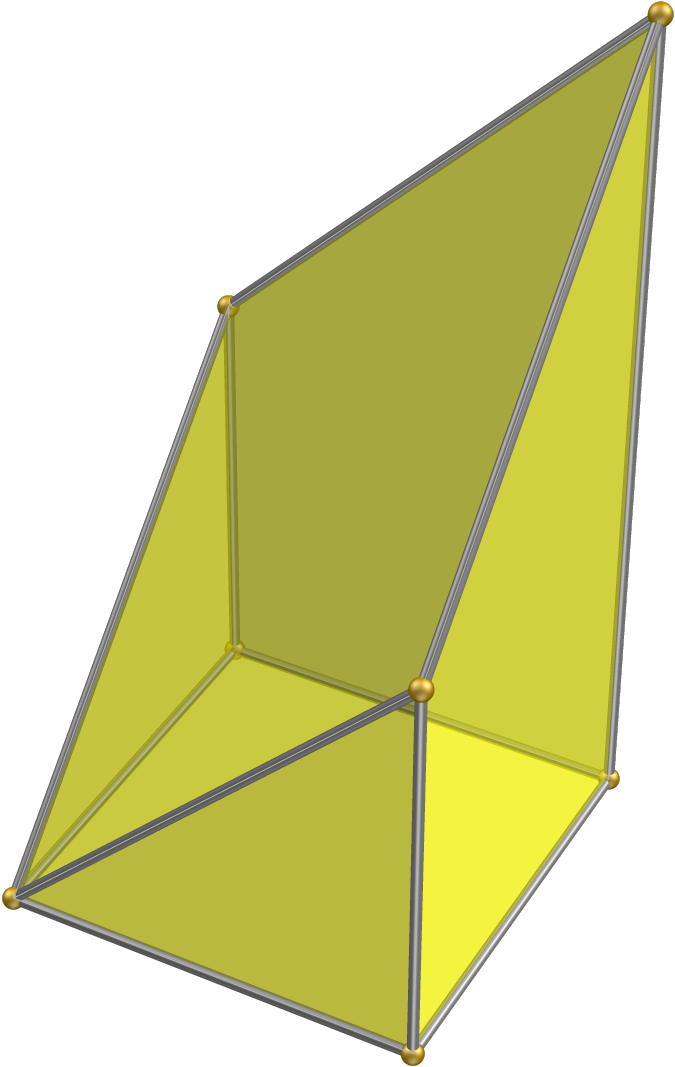}
\caption{The weight lifting polytope $\mathcal{P}_w$ of the unit square $\mathcal{P}=[0,1]^2$ relative to $w=y_1+y_2$. 
The weight lifting polytope $\mathcal{P}^w$ appears as the upper facet.}
\label{figure-vila}
\end{figure}
\end{example}  

\section{Open Problems}

We conclude the paper with a few open problems about weighted Ehrhart
functions and series. 

For $s$-weighted Ehrhart series, there are sufficient conditions on
the homogeneous weight function for the nonnegativity of the
coefficients of the $h^*$-polynomial \cite[Theorem 2.6,
Section~3.2]{jdl3}. But it remains an open problem to fully
characterize the family of homogeneous weights that yield nonnegative
coefficients of the $h^*$-polynomial.      

There are also natural questions about the properties of the
numerator in an $s$-weighted Ehrhart series. Are there sometimes
unimodal in their coefficients? For instance,  Stanley conjectured that
for lattice polytopes with the integer decomposition property they
have unimodal $h^*$-vectors (see \cite{Beck-J-M} and the survey paper
on unimodality for Ehrhart $h^*$-vectors \cite{Braun}), so it is natural
to ask when we add weights. What are the geometric conditions on the
polytope and the weights that force unimodality?      
Similarly, when do the coefficients of the $h^*$-vector admit
combinatorial interpretations? Or, what replaces the standard
Stanley/Hibi \cite{Hibi,Stanley-nonneg-h-vector,Sta2} inequalities in the weighted case? When are the
coefficient sequences log-concave? This is open even for many natural
weight functions.    

While some of the arguments and statements here only hold for lattice
polytopes it would be desirable to extend to rational polytopes. In
particular we ask whether  one can extend the reciprocity Theorem
\ref{Thm:s-weightedreciprocity} to half-opened polytopes. We
conjecture this is true and can have interesting consequences.
Ultimately, what is the correct reciprocity statement for a broad
classes of weights? If the weights are piece-wise polynomial, how do
singularities or discontinuities in the weight affect reciprocity?       

Another exciting new recent variation of weighted Ehrhart theory comes 
from playing with new types of weights, not necessarily coming from
polynomial functions, but instead weights that take values on a
general Abelian group (see \cite{ehrhartoverabeliangroup}).
Similarly, while for ordinary Ehrhart theory of lattice-point
counting, one has local formulas involving faces, cones, and  
Euler--Maclaurin phenomena (on the style of McMullen or
Berline--Vergne local formulas). See \cite{barvinokbook}. Can every
weighted Ehrhart coefficient be expressed canonically as a sum of
local face contributions? Are there weighted local formulas?   
It seems our weight lifting polytope may give a clue as to what these
local formulas can be. 

Finally, in this paper, we discussed earlier in Theorem
\ref{compatible-theorem} that some special compatible triangulations
suffice to give the positivity of the $q$-weighted Ehrhart series,
but it would be desirable to have a better picture and perhaps arrive
at a sufficient and necessary condition. Perhaps it is at least
possible to give a classification of those polytopes that have
compatible triangulations.

\section*{Acknowledgments} 
We used \textit{Normaliz} \cite{normaliz2} and \textit{Macaulay}$2$
\cite{mac2} to compute weighted Ehrhart polynomials and series and
volumes of lattice polytopes. We thank the anonymous referees for
their help and suggestions.   

\section*{\ }  
On behalf of all authors, the corresponding author states that there is no conflict of interest.

Data sharing is not applicable to this article as no datasets were
generated or analyzed during the current study. 

\bibliographystyle{plain}

\end{document}